\documentclass[11pt,a4paper,fleqn]{article}
\pdftrailerid{}
\usepackage[a4paper]{geometry}
\usepackage{graphicx}
\usepackage{amsmath,amssymb,latexsym,graphics,epsfig}
\usepackage{hyperref}
\usepackage{color}
\usepackage{amsthm}

\usepackage{tikz}

\usepackage{setspace}

\usepackage[space]{cite}
\usepackage{fullpage}

\newtheorem{theorem}{\bf Theorem}[section]

\newtheorem{lemma}[theorem]{\bf Lemma}
\newtheorem{corollary}[theorem]{\bf Corollary}
\newtheorem{conjecture}[theorem]{\bf Conjecture}

\newtheorem{remark}[theorem]{\bf Remark}

\newtheorem{definitiona}[theorem]{\bf Definition}

\newenvironment{definition}{\begin{definitiona}
\rm 
}{\end{definitiona}}

\DeclareMathOperator{\ex}{ex}
\DeclareMathOperator{\cl}{cl}

\numberwithin{equation}{section}

\let\hat\widehat

\usepackage[mathlines]{lineno}
\usepackage{etoolbox} %

\newcommand*\linenomathpatch[1]{%
   \expandafter\pretocmd\csname #1\endcsname {\linenomath}{}{}%
   \expandafter\pretocmd\csname #1*\endcsname{\linenomath}{}{}%
   \expandafter\apptocmd\csname end#1\endcsname {\endlinenomath}{}{}%
   \expandafter\apptocmd\csname end#1*\endcsname{\endlinenomath}{}{}%
 }
\newcommand*\linenomathpatchAMS[1]{%
    \expandafter\pretocmd\csname #1\endcsname {\linenomathAMS}{}{}%
    \expandafter\pretocmd\csname #1*\endcsname{\linenomathAMS}{}{}%
    \expandafter\apptocmd\csname end#1\endcsname {\endlinenomath}{}{}%
    \expandafter\apptocmd\csname end#1*\endcsname{\endlinenomath}{}{}%
}

\expandafter\ifx\linenomath\linenomathWithnumbers
\let\linenomathAMS\linenomathWithnumbers
\patchcmd\linenomathAMS{\advance\postdisplaypenalty\linenopenalty}{}{}{}
\else
\let\linenomathAMS\linenomathNonumbers
\fi

\linenomathpatchAMS{gather}
\linenomathpatchAMS{multline}
\linenomathpatchAMS{align}
\linenomathpatchAMS{alignat}
\linenomathpatchAMS{flalign}

\begin{document}

\title{\Large The Multicolor Induced Size-Ramsey Number of Long
  Subdivisions\footnote{Research partially supported by FAPESP
  	(2023/03167-5), CAPES (Finance Code 001) and INSF (4032328)}}

\date{}

\author{Ramin Javadi\thanks{Department of Mathematical Sciences,
    Isfahan University of Technology, Isfahan, 84156-83111,
    Iran. E-mail: rjavadi@iut.ac.ir. This work is based upon research
    funded by Iran National Science Foundation (INSF) under project
    No.~4032328} 
  \and
  Yoshiharu Kohayakawa\thanks{Instituto de Matem\'atica e
    Estat\'istica, Universidade de S\~ao Paulo, 05508-090 S\~ao
    Paulo, Brazil.  E-mail: yoshi@ime.usp.br.  Partially supported
    by CNPq (407970/2023-1, 420838/2025-2, 315258/2023-3)}
  \and Meysam Miralaei\thanks{Instituto de Matem\'atica e
    Estat\'istica, Universidade de S\~ao Paulo, 05508-090 S\~ao Paulo,
    Brazil.  E-mail: m.miralaei@ime.usp.br.  Supported by FAPESP
    (2023/04895-4)}}

\maketitle

\begin{abstract}
For a positive integer $k$ and a graph $H$, the $k$-color induced size-Ramsey number \linebreak $\widehat{R}_{\mathrm{ind}}(H, k)$ is the minimum 
integer $m$ for which there exists a graph $G$ with $m$ edges such that 
for every $k$-edge coloring of $G$, the graph $G$ contains a monochromatic copy of $H$ as an induced subgraph.  For a graph $H$ with the edge set $E(H)$ and a function $\sigma:E(H)\to \mathbb{N}$, the subdivision $H^\sigma$ is obtained by replacing each $e \in E(H)$ with a path of length $\sigma(e)$. 
We prove that for all integers $k,\,  D\geq 2 $, there exists a
constant $c=c(k, D)$ such that the following holds. Let $ H $ be any
graph with maximum degree~$D$ and let~$H^{\sigma}$ be a subdivision
of $H$ with $\sigma(e) > c \log_D n $ for every $e
\in E(H)$, where~$n$ is the order of~$H^\sigma$. Then,   $\hat{R}_{\mathrm{ind}}(H^\sigma,k)=e^{O(k\log k)} D^{9}\log (D)\, n$. 
If each $\sigma(e)$ is even and larger than $c \log_D n$, this bound improves to  $\hat{R}_{\mathrm{ind}}(H^\sigma,k)=O(k^{342} \log^{9} (k) D^{9} \log D )n$. 
We also find improved bounds for the non-induced size-Ramsey number of long subdivisions. 

\end{abstract}

\section{Introduction}	
For given graphs $H$ and $G$, and a positive integer $k$, we say that $G$ is $k$-\emph{Ramsey for} $H$, denoted by $G \longrightarrow (H)_k$, if for every edge coloring of $G$ with $k$ colors, there exists a monochromatic subgraph of $G$ isomorphic to $H$.
Ramsey's pioneering work~\cite{Ramsey} established that for any graph
$H$ and any positive integer $k$, there exists a  graph $G$ such that
$G \longrightarrow (H)_k$. Such a graph $G$ is called a $k$-\emph{Ramsey
  graph} for $H$. The minimum number of vertices of such a graph $ G
$, which clearly
can be taken to be a complete graph, is  the \emph{Ramsey number} $R(H, k)$. In other words,
\[
R(H ,k) = \min \{ |V(G)| : G \longrightarrow (H)_k \}.
\]

While this perspective focuses on minimizing the number of vertices, an alternative viewpoint—initiated by Erd\H{o}s, Faudree, Rousseau, and Schelp~\cite{S.R.Erd}—seeks to minimize the number of edges. This leads to the notion of the \emph{size-Ramsey number}. For a graph $H$ and integer $k \geq 2$, the \emph{multicolor size-Ramsey number} $\widehat{R}(H,k)$ is the minimum number of edges in any graph $G$ that is $k$-Ramsey for $H$. More formally, 
\[
\widehat{R}(H, k) = \min \{ |E(G)| : G \longrightarrow (H)_k \}.
\]

One of the central problems in this area is to determine for which
families of graphs the size-Ramsey number grows linearly with the
number of vertices. Erd\H{o}s~\cite{Erd_1} asked this question in the
case of paths. Beck~\cite{Beck_1} proved that $\widehat{R}(P_n, 2) <
900n$, that is, the size-Ramsey number of paths is linear. This line
of work has since been extended to various graph families:
cycles~\cite{Sudakov-cycle, Haxell-cycle, JKhOP, JM}, bounded-degree
trees~\cite{Friedman, Haxell-tree, Xin}, powers of paths and
cycles~\cite{Power_path, jie}, bounded powers of bounded-degree trees
(equivalently graphs of bounded treewidth and maximum
degree)~\cite{Power_tree, treewidth, jiang13:_degree_ramsey}, and logarithmic subdivisions of bounded-degree graphs~\cite{Rolling}. On the negative side, R\"{o}dl and Szemer\'{e}di~\cite{Rodl} constructed graphs on $n$ vertices with maximum degree $3$ and size-Ramsey number at least  $n \log^cn$, for a small constant $c>0$. This lower bound has  been
improved to $ cne^{c\sqrt{\log n}} $ for some $ c > 0 $ by Tikhomirov~\cite{Tikh}. Thus, 
the linearity of size-Ramsey numbers does not extend universally to bounded degree graphs.

A parallel research direction concerns the dependence of $\widehat{R}(H, k)$ on the number of colors~$k$. For paths $P_n$ it is known that $ \widehat{R}(P_n, k) = \Theta( k^2 \log k) n $ (see~\cite{New-path} for the lower bound and~\cite{Dudek.2, Krivel:upper} for the upper bound).

For cycles,  Haxell, Kohayakawa, and \L uczak~\cite{Haxell-cycle}, using regularity methods, proved that $\widehat{R}(C_n, k) \leq Cn$ where $C$ is a constant of tower type on $k$. Subsequently, Javadi and Miralaei~\cite{JM}  established that $\widehat{R}(C_n, k)=O(k^{120} \log^2 k)\, n$ for even $n$  and 
$  \widehat{R}(C_n, k)=\Omega(2^{k})\, n=O(2^{16 k^2+2\log k})\, n $ for odd $n$. This result improved a result by 
Javadi, Khoeini, Omidi, and Pokrovskiy~\cite{JKhOP}.  More recently, Brada\v{c}, Dragani\'c, and Sudakov~\cite{Sudakov-cycle} resolved the odd cycle case by proving $\widehat{R}(C_n, k) = e^{O(k)}n$.
They also improved the even cycle case by proving that $\widehat{R}(C_n, k)= O(k^{102})n$.

It is natural to consider the induced analogues of the above
quantities. For given graphs $H$ and $G$, and a positive integer $k$,
we say that $G$ is \emph{induced Ramsey for} $H$, denoted by
$G \longrightarrow_{\mathrm{ind}} (H)_k$, if for every $k$-edge coloring
of $G$, there exists a monochromatic induced subgraph of $G$ isomorphic to $H$. The \textit{multicolor induced Ramsey
  number} $R_{\mathrm{ind}}(H, k)$ and the \textit{multicolor
  induced size-Ramsey number} $\widehat{R}_{\mathrm{ind}}(H, k)$ are defined to be the minimum number of vertices (respectively
edges) in a graph $G$ such that
$G\longrightarrow_{\mathrm{ind}}(H)_k$. The existence of
Ramsey-induced graphs was proved independently by Deuber \cite{Deub},
Erd\H{o}s, Hajnal, and P\'osa \cite{Erd-Hajnal-Posa} and
R\"odl~\cite{Rodl-induced}.  Although the size-Ramsey number of graphs
has been studied for a long time, less is known about the induced
version. By a result of Dragani\'c, Krivelevich and Glock
\cite{Dra-Kri-Glo}, we know that
$\widehat{R}_{\mathrm{ind}}(P_n, k)= O(k^3 \log^4 k)n$.  The
above-mentioned result of Haxell, Kohayakawa and \L
uczak~\cite{Haxell-cycle} is in fact proved in the induced setting,
and hence $\widehat{R}_{\mathrm{ind}}(C_n, k)\leq Cn$, where~$C$ has a
tower type dependence on~$k$.  This result was improved recently by
Brada\v{c}, Dragani\'c, and Sudakov~\cite{Sudakov-cycle}, who proved
that if $n$ is even, then
$\widehat{R}_{\mathrm{ind}}(C_n, k) = O(k^{102})n$ and if $n$ is odd,
then $\widehat{R}_{\mathrm{ind}}(C_n, k) = e^{O(k\log k)}n$.
Recently, Gir\~ao and Hurley~\cite{Girao} proved that for every tree
$T$ with maximum degree $\Delta$, we have
$ \widehat{R}_{\mathrm{ind}}(T, k)=O_{k,\Delta}(n)$, where the implied
constant in the big-$O$ notation is a small polynomial in~$k$
and~$\Delta$.  A remarkable result of Hunter and
Sudakov~\cite{hunter25:_induc_ramsey} establishes that
$\widehat{R}_{\mathrm{ind}}(H, k)=O_{\Delta,w,k}(n)$ for any graph~$H$
with $\Delta(H)\leq\Delta$ and treewidth at most~$w$, where the
implied constant is quite large, owing to the use of regularity
arguments (they also give an alternative approach that leads to
smaller constants in the case of trees and blow-up of trees).
Finally, Arag\~ao, Campos, Dahia, Filipe and
Marciano~\cite{aragao25:_ramsey} proved that
$\widehat{R}_{\mathrm{ind}}(H, k)=e^{O(k\log k)n}$ for any $n$-vertex
graph~$H$, confirming a conjecture of Erd\H{o}s \cite{Erd-ind} in a
strong form.  In this paper, we focus on estimating the induced
size-Ramsey number of long subdivisions of bounded degree graphs.

Given a graph $H$ and a function $\sigma:E(H)\to \mathbb{N}$, the \emph{subdivision} $H^\sigma$ is obtained by replacing each edge $e \in E(H)$ with a path of length $\sigma(e)$.
Pak~\cite{Pak} studied the size-Ramsey number of  long subdivisions of bounded degree
graphs, which, roughly speaking, are subdivided
graphs $ H^{\sigma} $ where $ \sigma(e)=\Omega(\log |V(H^{\sigma})|) $  for every $e\in E(H)$ and the maximum degree of $ H $ is bounded. He posed the following conjecture.

\begin{conjecture}{\rm~\cite{Pak}}\label{conj:pak}
	For every $ k,\, D\in \mathbb{N}  $ there exist  $ C,\, L > 0 $ such that if $ H $
	is a graph with $ \Delta(H) \leq D $ and $ \sigma(e)\ge L\log(|V(H^{\sigma})|) $ for all $ e \in E(H) $, then
	$\widehat{R}(H^{\sigma}, k)\leq C|V(H^{\sigma})| $.
\end{conjecture}

Pak \cite{Pak} came close to proving Conjecture~\ref{conj:pak}: he
proved that
\[
\widehat{R}(H^{\sigma}, k)=O(|V(H^{\sigma})|\log^3|V(H^{\sigma})|).
\]
In the special case where $ H $ is a fixed graph and $ \sigma(e) $ tends to infinity for every $e\in E(H)$, the conjecture was verified by Donadelli, Haxell, and Kohayakawa~\cite{DHK}.

Finally, Dragani\'c, Krivelevich, and Nenadov~\cite{Rolling} confirmed Conjecture~\ref*{conj:pak}. 
Their proof is based on  Szemer\'edi's regularity lemma and makes no
attempt to optimize the constants~$C$ and~$L$.

In this paper, we prove Conjecture~\ref{conj:pak} in the induced
setting with the goal of obtaining good explicit bounds on the
constants, in particular avoiding the use of the regularity lemma.  Our
main theorem provides upper bounds on
$\widehat{R}_{\mathrm{ind}}(H^\sigma, k)$ with constants that are
exponential, and in certain cases polynomial, in~$k$ and~$D$,
depending on the structural restrictions imposed on the subdivision.

\begin{theorem}
  \label{thm:induced_subd_intro}
  Let $k,\,D\geq 2$ be integers and $H^\sigma$ be a subdivision of
  a graph $H$ with maximum degree $D$ and let $n=|V(H^\sigma)|$.  Then
  there exists an integer $n_0=n_0(k,D)$ such that for every
  $n\geq n_0$ the following hold.
  \begin{itemize}
  \item If $\sigma(e)$ is an even integer larger than
    $8\log_{D}n+Ck^7\log^2(kD)$ for any $e\in E(H)$, where~$C$ is a
    suitably large absolute constant, then
    \begin{equation*}
      \hat{R}_{\mathrm{ind}}(H^\sigma,k)={O(k^{342}
        \log^{9} (k) D^{9} \log D )n}.
    \end{equation*}

  \item If $\sigma(e)$ is an arbitrary integer larger than
    $12\log_{D}n+e^{Ck\log k}\log^2 D$ for any $e\in E(H)$, where~$C$
    is a suitably large absolute constant, then
    \begin{equation*}
      \hat{R}_{\mathrm{ind}}(H^\sigma,k)={e^{O(k\log k)}
        D^{9}(\log D)\, n}.
    \end{equation*}
  \end{itemize} 
\end{theorem}

Our proof follows the approach of Brada\v{c}, Dragani\'c and
Sudakov~\cite{Sudakov-cycle}: let $\mathcal{H}$ be an $s$-uniform
hypergraph with $N$ vertices and $CN$ random hyperedges, $C=C(k, D)$,
with some appropriate properties including a sparsity condition on
small subsets and linearity (every two hyperedges share at most one
vertex).  Then our host graph (the induced Ramsey graph proving our
bounds) is the graph obtained from $\mathcal{H}$ by replacing every
hyperedge of $\mathcal{H}$ with a copy of a gadget graph~$F$ which has
$s$ vertices and is an induced-Ramsey graph for some short cycles.

To work in the induced set-up, we follow the approach of Gir\~ao and
Hurley~\cite{Girao}, where an induced extension setting of the
Friedman--Pippenger theorem on tree embeddings is obtained by applying
a certain “\textit{pre-emptive greedy algorithm}”.  Suppose that we
wish to embed a tree in a host graph as an induced subgraph by
successively adding pendant edges.  During the embedding, some
vertices appear whose neighbors are mostly unusable since either these
neighbors are already used or adding them violates “inducedness”.  We
call these vertices “critical vertices”.  Following ideas
from~\cite{Girao}, we pre-empt this situation simultaneously
reserving some neighbors for these critical vertices, to be used
in future if required.

Before closing this section, we remark that, although the main focus
of this paper is on induced size-Ramsey numbers, our approach also
leads to a new proof of Conjecture~\ref{conj:pak}, without the use of
regularity methods.  Our proof leads to the numerical constants given
in the theorem below.

\begin{theorem}
  \label{thm:noninduced_subd_intro}
  Let $k,\,D\geq 2$ be integers and $H^\sigma$ be a subdivision of
  a graph $H$ with maximum degree $D$ and let $n=|V(H^\sigma)|$.
  Then there exists an integer $n_0=n_0(k,D)$ such that for
  every $n\geq n_0$ the following hold.
  \begin{itemize}
  \item If $\sigma(e)$ is an even integer larger than
    ${8\log_{D}n+Ck^3 \log^2 D}$ for any
    $e\in E(H)$, where~$C$ is a suitably large absolute constant, 
    then
    \begin{equation*}
      \hat{R}(H^\sigma, k)={O(k^{34}
        D^5\log D)n}
    \end{equation*}

  \item If $\sigma(e)$ is an arbitrary integer larger than
    ${2\log_2^2k\log_{D}n+Ck^4 2^{2k}\log^2 D}$ for
    any $e\in E(H)$, where~$C$ is a suitably large absolute constant, 
    then
    \begin{equation*}
      \hat{R}(H^\sigma,k)={O(2^{34k}
        k^6 \log^5(k) D^5\log D)n}      
    \end{equation*}
  \end{itemize}
\end{theorem}

\subsection*{Conventions and notation}
The set of integers $\{1,\ldots, n\}$ is denoted by $[n]$.  For a
graph $ G $, we write $ V(G) $, $ E(G) $, $v(G)$ and $ e(G) $ for the
vertex set, edge set, the number of vertices and the number of edges
of $ G $, respectively.  For $ v \in V(G) $, by $ N_{G}(v) $ we mean
the set of all neighbors of $ v $ in $G$ and the degree of $ v $ is
$ d_{G}(v)=\vert N_{G}(v) \vert $. Vertices with degree one are called \emph{leaves} of $G$.  Also, the average degree of $G$ is
$\bar{d}(G)=2e(G)/v(G)$.  For a subset $ S\subseteq V(G) $, we define
the neighborhood of $S$ as
$ N_G(S)=\big(\bigcup_{u\in S}N_G(u)\big) \setminus S $. Also, the induced
subgraph~$G[S]$ of $G$ on $S$ is the graph obtained from $G$ by
deleting all vertices in $V(G)\setminus S$.  
For a number $\gamma>0$, we say that the graph $G=(V,E)$ is a
\textit{$\gamma$-expander} if for every subset $S\subseteq V$ with
$|S|\leq |V|/2$, we have $|N_G(S)| \geq \gamma |S|$.

For a positive integer
$D$, a full $D$-ary tree is a rooted tree whose all non-leaf vertices
have $D$ children.  For a rooted tree $T$, the distance of a vertex
$v\in V(T)$ from the root is called the depth of $v$ and the maximum
depth of a vertex in $V(T)$ is called the height of $ T $. If a rooted
tree has only one vertex (the root), its height is zero.
If $ A,B \subseteq V(G) $, then
$ E_G(A,B)= \{ xy\in E(G) : x\in A,\, y\in B\} $ is the set of edges
connecting a vertex of $ A $ to a vertex of $ B $. Also,
$e_G(A,B)=|E_G(A,B)|$. We drop the subscript $G$ whenever there is no
ambiguity.

A hypergraph is said to be \emph{linear} if any two hyperedges
intersect in at most one vertex.  A (\emph{Berge})
\emph{cycle} of \emph{length}~\(k\) in \(\mathcal{H}\), where
\(k \geq 2\), is a sequence $v_1 E_1 v_2 E_2 \ldots v_k E_k$ of
distinct vertices \(v_1, \ldots, v_k\) and distinct hyperedges
\(E_1, \ldots, E_k\) of \(\mathcal{H}\) such that
\(v_1 \in E_1 \cap E_k\) and \(v_i \in E_{i-1} \cap E_i\) for all
\(2\leq i \leq k\).  The \emph{girth} of \(\mathcal{H}\), denoted
\(\mathrm{girth}(\mathcal{H})\), is the length of the shortest cycle 
in \(\mathcal{H}\). In particular, any hypergraph $\mathcal{H}$ with
$\mathrm{girth}(\mathcal{H})\ge 3$ is linear.

Given functions $f(n), \, g(n)$ and $h(n)$, we write
$f(n)=\Omega(h(n))$ (resp.\ $f(n)=O(g(n))$) if there exist absolute constant
$ C>0 $ such that $f(n)\geq Ch(n)$ (resp.\ $f(n)\leq Cg(n)$) for all
sufficiently large~$n$.  All logarithms are taken to base $ e $,
unless stated. Throughout the paper, we omit floor and
ceiling symbols whenever they are not essential.

\section{Tools}

In this section, we provide some tools that will be used in the follow-up sections.

An $s$-uniform hypergraph $\mathcal{H}$ is called \emph{bipartite} if $V(\mathcal{H})$ can be partitioned into two sets $(X,Y)$ such that any hyperedge in $E(\mathcal{H})$ intersects $X$ in exactly one vertex. 
For a subset $I\subseteq X$, we define the \emph{neighbourhood} $N_{\mathcal{H}}(I)$ of $I$ as follows:
\begin{equation*}
  N_{\mathcal{H}}(I)=\{ A\subseteq Y: A\cup \{x\} \in E(\mathcal{H})
  \text{ for some } x\in I\}.
\end{equation*}
In a bipartite hypergraph $\mathcal{H}$ with bipartition $(X,Y)$, a \emph{matching} saturating $X$ is a collection of hyperedges  which are mutually disjoint and their union contains $X$. 
For a collection $\mathcal{Y}$ of subsets of $Y$, a \textit{transversal} of $\mathcal{Y}$ is a subset $T\subseteq Y$ which intersects all members of $\mathcal{Y}$. The minimum size of a transversal of $\mathcal{Y}$ is denoted by $\tau(\mathcal{Y})$. 
The following theorem by Haxell \cite{haxell} provides a sufficient condition for a bipartite hypergraph to have a matching saturating $X$.

\begin{theorem} {\rm \cite{haxell}} \label{thm:aharoni}
Let $s\geq 2$ be an integer and $\mathcal{H}$ be an $s$-uniform bipartite hypergraph with bipartition $(X,Y)$. Suppose that for every $I\subseteq X$, we have
\[
\tau(N_{\mathcal{H}}(I))\geq (2s-3) (|I|-1)+1.
\]
Then, $\mathcal{H}$ admits a matching saturating $X$.
\end{theorem}

We need a simple corollary of the above theorem. Given an $s$-uniform bipartite hypergraph $\mathcal{H}$ with bipartition $(X,Y)$ and a positive integer $D$, by a \textit{$D$-matching} saturating $X$, we mean a collection of hyperedges $\mathcal{M}\subseteq E(\mathcal{H}) $ such that each vertex $x\in X$ is in exactly $D$ hyperedges of $\mathcal{M}$ and the sets $M\setminus X$ $(M\in \mathcal{M})$ are mutually disjoint. 

\begin{corollary} \label{cor:aharoni}
	Let $s\geq 2$ and $D$ be positive integers and $\mathcal{H}$ be an $s$-uniform bipartite hypergraph with bipartition $(X,Y)$. Suppose that for every $I\subseteq X$, we have
	\[
	\tau(N_{\mathcal{H}}(I))\geq (2s-3) (D|I|-1)+1.
	\]
	Then, $\mathcal{H}$ admits a $D$-matching saturating $X$.
\end{corollary}
The above corollary follows immediately from Theorem~\ref{thm:aharoni} by replacing each vertex $x$ of $X$ with $D$ new vertices $x_1,\ldots, x_D$ and replacing each hyperedge containing $x$ with $D$ new hyperedges each containing one $x_i$. 

We use the following version of Chernoff's bound in proving the next lemma.
\begin{theorem}{\rm (Chernoff's bound)} \label{thm:chernoff}
Let $X$ be a binomial random variable with $\mathbb{E}(X)=\mu$ and $0\leq \delta\leq 1$. Then, 
\begin{equation}
\mathbb{P}\big[X\leq (1-\delta)\mu\big]\leq \exp\left({\frac{-\delta^2\mu}{2}}\right).
\end{equation}
\end{theorem}

The following lemma is a generalization of a result by Alon and
Spencer (see Proposition~5.5.3 in~\cite{alon-spencer}), which provides
a sufficient condition for the existence of a large independent set in
a multipartite graph.  The relevance and applicability of a result
such as Lemma~\ref{lem:LLL} below in induced embedding problems was
nicely highlighted in the proof of~\cite[Lemma~2.13]{Girao}.
Lemma~\ref{lem:LLL} will be used in the proof of
Lemma~\ref{lem:extend-induced} in Section~\ref{sec:extending}, which
parallels~\cite[Lemma~2.13]{Girao} in our set-up.

\begin{lemma} \label{lem:LLL}
Let $d,\, a\geq 3$ and $t\geq 2$ be some positive integers and $G$ be a $t$-partite graph with $t$-partition $(A_1,\ldots, A_t)$ such that $|A_i|=a$ for each $i\in [t]$. Also, suppose that 
\begin{itemize}
	\item $G$ is $d$-degenerate, 
	\item for each $i,j\in [t]$, $i\neq j$, we have $e_G(A_i,A_j)\leq 1$, and
	\item the maximum degree $\Delta_G$ of $G$ satisfies $\Delta_G\leq \exp({a/(200d)})$. 
\end{itemize}
 Then, for each $i\in [t]$, there exists a subset $A'_i\subseteq A_i$ with $|A'_i|\geq  a/(100d)$ and $\bigcup_{i\in [t]} A'_i$ is a stable set in $G$.
\end{lemma}
\begin{proof}
First, we orient the edges of $G$ in such a way that the in-degree of
each vertex is at most $d$ (which can be done since $G$ is
$d$-degenerate).  Let~$\vec G$ be this oriented graph.
Now, we randomly and independently select each vertex $v\in V(G)$ with probability $p={1}/{d}$. Let $S$ be the set of chosen vertices. Now, if there exists an edge from $u$ to $v$ for some vertices $u,v \in S$, we remove $v$ from $S$. Let $S'$ be the set of remaining vertices in $S$ and $A_i'= S'\cap A_i$, for all $i\in [t]$. It is evident that $\bigcup_{i\in [t]} A'_i$ is a stable set in $G$. For a vertex $v\in V$, if $v\in S'$, then $v\in S$ and all its in-neighbors are not in $S$. Then, we have
\[ \mathbb{P}(v\in S')\geq p(1-p)^d \geq p\, \exp\left({\frac{-pd}{1-p}}\right)\geq \frac{1}{e^{2}d}. \]
Now, note that since $e(A_i,A_j)\leq 1$, every two vertices $u,v$ in $A_i$ have no common in-neighbor. Therefore, events $(u\in A_i')$ and $(v\in A_i')$ are independent and thus  $|A'_i|$ is a binomial random variable with $\mathbb{E}(|A'_i|)\geq e^{-2}d^{-1}a$. Now, applying Chernoff's inequality (Theorem~\ref{thm:chernoff}) with $\delta=9/10$, we have 
\[\mathbb{P}\left(|A'_i|\leq \frac{a}{100d}\right)\leq \mathbb{P}\left(|A'_i|\leq \frac{a}{10e^2 d}\right)\leq \exp\left(\frac{-81a}{200e^2d}\right)\leq \exp\left(\frac{-a}{20d}\right). \]

Now, we apply Lov\'asz's Local Lemma (LLL) as follows. For each
$i\in [t]$, let $B_i$ be the event that $|A'_i|<a/(100d)$.  Also, for
$i\in[t]$, let $S_i=A_i\cup\bigcup_jA_j$, where the union ranges
over~$j\in[t]$ such that there is an arc from a vertex in~$A_j$ to
some vertex in~$A_i$ in~$\vec G$.  We define a graph~$\Gamma$ on the
$B_i$ ($i\in[t]$) by joining~$B_i$ and~$B_j$ with an edge if and only
if $S_i\cap S_j\neq\emptyset$.  One can check that~$\Gamma$ is a
dependency graph for the events~$B_i$, and that
$\Delta=\Delta(\Gamma)\leq a\Delta_G^2$.  Also, since $d\leq \Delta_G$, we
have $a\geq 200\,d\ln d$. Now,
\begin{align*}
e\, \mathbb{P}(B_i)\,\Delta & \leq  \exp{\left(1-\frac{a}{20d}\right)} a\exp\left({\frac{a}{100d}}\right)
\\&=a  \exp{\left(1-\frac{a}{25d}\right)}\\
&\stackrel{(*)}{\leq} 200\,d\ln d\, \exp{(1-8\ln d)}\\
&=200\ln d\, \frac{e}{d^7}<1,
\end{align*} 
where Inequality ($*$) holds because $a\geq 200d\ln d$ and the function $f(a)=a\exp(1-a/(25d))$ is decreasing with respect to $a$. 
Hence, by LLL, $\mathbb{P} (\bigcap_{i} \overline{B_i}) >0$ and the proof is complete.  
\end{proof}

\section{Host graph}
Let $\mathcal{H}$ be an $s$-uniform linear hypergraph. Also, let $F$ be a graph on $s$ vertices. \textit{The substitution} $\mathcal{H}(F)$ of $F$ in $\mathcal{H}$, is a graph whose vertex set is the vertex set of $\mathcal{H}$ and every hyperedge of $\mathcal{H}$ is replaced with a copy of $F$. 

Our host graph is a graph $\mathcal{H}(F)$, where $\mathcal{H}$ has a linear number of hyperedges with some properties including sparsity condition on small subsets (see Conditions (P4) and (P4$'$) in the lemma below) and $F$ is a Ramsey (or induced-Ramsey) graph for some cycles (see Table~\ref{tbl:gadgets}). 

First, we need the following lemma which guarantees the existence of a hypergraph with some specific properties. It is a modified version of Lemma~4.4 in \cite{Sudakov-cycle}.  Indeed in \cite[Lemma~4.4]{Sudakov-cycle} it was shown that if $0<\alpha\leq 10^{-3}c^{-2}s^{-4}$, then there exists an $s$-uniform hypergraph $\mathcal{H}'$ satisfying (P1)--(P3) and (P4)$'$. In the following lemma we replace the Property (P4)$'$ with a new Property (P4), however, this strengthened property holds only under the more restrictive condition $0<\alpha \leq 10^{-4} c^{-5}s^{-15}$. 

\begin{lemma} \label{lem:hypergraph}
Let $s, g\geq  4$ and $c\geq 1$ be given integers and set $0<\alpha \leq 10^{-4} c^{-5}s^{-15}$. Then, there exists a constant $N_0:=N_0(s,g,c, \alpha)$ such that for any integer $N\geq N_0$, there is an $s$-uniform hypergraph $\mathcal{H}$ on $N$ vertices, such that the following properties hold.
\begin{itemize}
\item[\rm (P1)] $e(\mathcal{H}) \in [c N/2,c N]$.
\item[\rm (P2)] $\Delta(\mathcal{H})\leq 8c s$.
\item[\rm (P3)] The girth of $\mathcal{H}$ is larger than $g$.
\item[\rm (P4)] For any   $A \subseteq V(\mathcal{H})$ with $|A|\leq \alpha N$, we have
$\sum_{h\in E(\mathcal{H})} |h\cap A|< \frac{5}{2}|A|$, 
where the sum is over all hyperedges $h\in E(\mathcal{H})$ with $|h\cap A|\geq 2$.
\end{itemize}
Also, if $0<\alpha\leq 10^{-3}c^{-2}s^{-4}$, then there exists an $s$-uniform hypergraph $\mathcal{H}'$ satisfying (P1)--(P3) along with the following property.
\begin{itemize}
\item[\rm (P4$'$)] For any   $A \subseteq V(\mathcal{H}')$ with $|A|\leq \alpha N$, there are at most $2|A|$ hyperedges in $\mathcal{H}'$ which intersect $A$ in at least two vertices.
\end{itemize}
\end{lemma}
\begin{proof}
We take $V=[N]$ and randomly and independently pick $m=\lfloor cN\rfloor $ subsets of size $s$ from $V$ as hyperedges of $\mathcal{H}$. Then, we remove duplicate hyperedges, a hyperedge from each Berge cycle of length at most $g$ and all hyperedges incident to a vertex of degree more than $8cs$. Conditions (P2) and (P3) hold trivially. Also, one can verify that Condition (P1) holds with high probability (the proof of (P1)--(P3) is the same as the proof of Lemma~4.4 in \cite{sudakov} and we omit it). We just prove that $\mathcal{H}$ satisfies (P4) with positive probability. 

Fix a subset $A\subseteq V$ with $|A|=a\leq \alpha N$. We say that $A$ is a bad subset if 
\[
\sum_{h\in E(\mathcal{H})} |h\cap A|\geq \frac{5}{2}a,
\]
where the sum is over all hyperedges $h\in E(\mathcal{H})$ with $|h\cap A|\geq 2$.
If $A$ is bad, then there is a vector of non-negative integers $\mathbf{a}=(a_2,\ldots, a_s)$, such that $2a_2+\cdots+sa_s=\lceil (5/2)a\rceil$ and for each $2\leq i\leq s$, there are $a_i$ hyperedges  in $ E(\mathcal{H})$ which intersect $A$ in exactly $i$ vertices. 
Also, define $\sigma(\mathbf{a}):= \sum_{i=2}^{s} a_i$. 
For each hyperedge $h\in E(\mathcal{H})$, we have
\[
\mathbb{P}(|h\cap A|=i)=  \frac{\binom{a}{i}\binom{N-i}{s-i}}{\binom{N}{s}} \leq \binom{a}{i} \left(\frac{s}{N}\right)^i \leq {\left(\frac{a s}{N}\right)^i}.
\]

Thus, 
\begin{align}
\mathbb{P}( A \text{ is bad}) &\leq \sum_{2a_2+\cdots+sa_s=\lceil (5/2) a \rceil} \binom{\lfloor cN \rfloor}{a_2}\binom{\lfloor cN \rfloor-a_2}{a_3}\ldots \nonumber \\
&\hspace{4cm}\binom{\lfloor cN \rfloor-a_2\cdots-a_{s-1}}{a_s} 
\left(\frac{a s}{N}\right)^{\sum_i ia_i} \nonumber\\
&=  \sum_{2a_2+\cdots+sa_s=\lceil (5/2) a \rceil}\dfrac{\lfloor cN \rfloor!}{a_2!\cdots a_s! (\lfloor cN \rfloor-a_2\cdots-a_s)!} \left(\dfrac{as}{N}\right)^{\frac{5}{2}a}\nonumber\\
&\leq  \sum_{2a_2+\cdots+sa_s=\lceil (5/2) a \rceil}\dfrac{(cN)^{\sigma(\mathbf{a})}}{a_2!\cdots a_s!} \left(\dfrac{as}{N}\right)^{\frac{5}{2}a}\nonumber\\
&\leq  \sum_{2a_2+\cdots+sa_s=\lceil (5/2) a \rceil}\dfrac{(ecN)^{\sigma(\mathbf{a})}}{({a_2})^{a_2}\cdots (a_s)^{a_s}} \left(\dfrac{as}{N}\right)^{\frac{5}{2}a}\label{eq:aa}\\
&\stackrel{*}{\leq}  \sum_{2a_2+\cdots+sa_s=\lceil (5/2) a \rceil}\left(\dfrac{(s-1)ecN}{\sigma(\mathbf{a})}\right)^{\sigma(\mathbf{a})} \left(\dfrac{as}{N}\right)^{\frac{5}{2}a}\nonumber\\
&\leq  \left(\frac{5}{4}a\right)^s  \left(\dfrac{4escN}{5a}\right)^{\frac{5}{4}a} \left(\dfrac{as}{N}\right)^{\frac{5}{2}a},\nonumber
\end{align}
where Inequality~$(*)$ holds because the denominator in \eqref{eq:aa} is $\exp(-H(a_2,\ldots, a_s))$ in which $H$ is the Entropy function and $H$ is maximized whenever all $a_i$'s are equal. 
Also, the last inequality holds because $\sigma(\mathbf{a})\leq 5a/4$ and the function $f(x)=((s-1)ecN/x)^x$ is increasing for $x\leq (s-1)cN$.  Therefore,

\begin{align*}
	\mathbb{P}(\exists A: A \text{ is bad}) &\leq \sum_{a=1}^{\alpha N} \binom{N}{a} \left(\frac{5}{4}a\right)^s \left(\dfrac{4escN}{5a}\right)^{\frac{5}{4}a} \left(\dfrac{as}{N}\right)^{\frac{5}{2}a}\\
	&\leq    \sum_{a=1}^{\alpha N} \left(\dfrac{Ne}{a}\right)^a \left(\frac{5}{4}a\right)^s \left(50s^{15} c^{{5}}(\frac{a}{N})^{5} \right)^{\frac{a}{4}}\\
	&\leq    \sum_{a=1}^{\alpha N}  \left(\frac{5}{4}a\right)^s \left(3000s^{{15}} c^{{5}}\frac{a}{N} \right)^{\frac{a}{4}}\\
		&=    \sum_{a=1}^{s^4-1}   \left(\left(\dfrac{5a}{4}\right)^{\frac{4s}{a}}3000s^{{15}} c^{{5}}\frac{a}{N} \right)^{\frac{a}{4}}+
		 \sum_{a=s^4}^{\alpha N}   \left(\left(\dfrac{5a}{4}\right)^{\frac{4s}{a}}3000s^{{15}} c^{{5}}\frac{a}{N} \right)^{\frac{a}{4}}
		\\
	&	\leq   \sum_{a=1}^{s^4-1}   \left(\left(\dfrac{5a}{4}\right)^{\frac{4s}{a}}3000s^{{15}} c^{{5}}\frac{a}{N} \right)^{\frac{a}{4}}+
		\sum_{a=s^4}^{\alpha N}  \left(6000s^{15} c^{5}\alpha  \right)^{\frac{a}{4}},\\
\end{align*}
where the last inequality holds because if $f(a,s)=(5a/4)^{(4s/a)}$, then $f(a,s)\leq f(s^4,s)\leq 2$, for all $a\geq s^4$. In the last line, the first summation tends to zero as $N$ goes to infinity and since  $6000 s^{15}c^5 \alpha <6/10$ and $s\geq 4$, the second summation is less than $1/4$. Hence, (P4) holds with  probability bigger than 3/4. 
\end{proof}
We apply the hypergraphs $\mathcal{H}$ and $\mathcal{H}'$ to prove upper bounds for the induced size-Ramsey number and the size-Ramsey number of a subdivision of a bounded degree graph, respectively (see Theorems~\ref{thm:induced_subd} and \ref{thm:noninduced_subd}). 
We also need the bounds on the order and size of the gadget graph~$F$. 

\begin{lemma}\label{lem:induced_even_odd_gadget} {\rm \cite{Sudakov-cycle}}
For every positive integer $k$,
there is a graph $F$ which is induced $k$-Ramsey for $C_6$ with $v(F)=O(k^6)$ and $e(F)=O(k^9)$. Also, there is a graph $F$ which is induced $k$-Ramsey for $C_5$ with $e(F)=e^{O(k\log k)}$.
\end{lemma}
\begin{lemma}\label{lem:even_gadget}
There exists a constant $C$, such that for every positive integer $k$, the complete graph on $\lceil Ck^{3/2}\rceil $ vertices is $k$-Ramsey for $C_6$. In particular, $R(C_6,k)=O(k^{3/2})$ and $\hat{R}(C_6,k)=O(k^3)$. 
\end{lemma}
\begin{proof}
Due to Bondy and Simonovits \cite{simonovits}, we have $\ex(n,C_6)=C'n^{4/3}$, for some constant $C'$. Let $F$ be the complete graph on $\lceil 9(C'k)^{3/2}\rceil$ vertices. Then, in every $k$-edge coloring of $F$, there exists a monochromatic subgraph with at least $e(F)/k> 9^{4/3}C'^3k^2=C' v(F)^{4/3}$ edges. Hence, $F$ contains a monochromatic copy of $C_6$.
\end{proof}
\begin{lemma} {\rm \cite{axenovich}} \label{lem:odd_gadget}
For every positive integer $k$ and every $k$-edge coloring of the complete graph on $4^k+1$ vertices, there exists a monochromatic odd cycle of length at most $2\lceil \log_2k\rceil+1$.
\end{lemma}

We divide the non-induced and induced size-Ramsey problems of subdivisions into two cases: the case that all subdivisions are of even length and the case that the subdivisions are of arbitrary length. 
In the case that all subdivisions are of even length, i.e. $\sigma(e)$ is even for all $e\in E(H)$, we take
$F$ as the graph which is $k$-induced-Ramsey (resp.\ $k$-Ramsey) for
the 6-cycle (given in Lemmas~\ref{lem:induced_even_odd_gadget} and
\ref{lem:even_gadget}). Also, in the general case, where the lengths
of subdivisions are odd or even, we take $F$ as the graph which is
$k$-induced-Ramsey (resp.\ $k$-Ramsey) for the 5-cycle (resp.\ an odd cycle of length at most $2\lceil \log_2 k\rceil +1$) given in Lemmas~\ref{lem:induced_even_odd_gadget} and~\ref{lem:odd_gadget} (see Table~\ref{tbl:gadgets}). 
Also, hereafter, the parameters in Lemma~\ref{lem:hypergraph} are chosen as in Table~\ref{tbl:parameters}.

\begin{table}
	\begin{center}
	\begin{tabular}{c|ccc}
		Case & Gadget graph $F$ &  $s=v(F)$ & $e(F)$ \\
		\hline
		Induced Case \\
        \cline{1-1}
		Even Subdivision &Induced $k$-Ramsey for $C_6$ & $O(k^6)$ & $O(k^9)$  \\
		 General Subdivision  & Induced $k$-Ramsey for $C_5$& $e^{O(k\log k)}$& $e^{O(k\log k)} $ \\
		 \hline 
		 Non-induced Case\\
         \cline{1-1}
		Even Subdivision & $k$-Ramsey for $C_6$ & $O(k^{3/2})$&  $O(k^3)$\\
		General Subdivision & complete graph on $4^{k}+1$ vertices & $4^{k}+1$& $2^{4k-1}+2^{2k-1}$ \\
			\end{tabular}
\end{center}
\vskip -6pt
	\caption{Gadget graph in the case of even and general subdivision.}\label{tbl:gadgets}
\end{table}

\begin{table}
	\begin{center}
		\begin{tabular}{c|p{0.3\linewidth}|p{0.4\linewidth}}
			Parameter & {Values for Induced version} & Values for Non-induced version \\
			\hline && \\[-2mm]
			Hypergraph & $\mathcal{H}$& $\mathcal{H}'$\\
			$c$ &$10^7ks^4\ln (s)D $& \parbox{\linewidth}{even case:~~~ $10^3ks^2D$\\ general case: $10^3(2\lceil \log_2 k \rceil +1) k s^2D$ } \\[4mm]
			$\alpha$ & $10^{-4} c^{-5}s^{-15}$ & $10^{-3}c^{-2} s^{-4}$ \\[2mm]
			$N$ &$\lceil 10^{13}  c^7s^{22} k \ln(sD) D \,  n\rceil $  & $\lceil 10^{13}c^3s^{9}k\ln (k)D\ln (sD)\, n\rceil $ \\[2mm]
			$g$& $\lceil 10^{10}sk\ln^2(sD)\rceil $ & $\lceil 10^{10} sk\ln (k) \ln (sD)\rceil $ 
		\end{tabular}
	\end{center}
	\vskip -6pt
	\caption{Setting values of parameters in applying  Lemma~\ref{lem:hypergraph}. }\label{tbl:parameters}
\end{table}

Our main result for the induced size-Ramsey is as follows.

\begin{theorem} \label{thm:induced_subd}
	Let $k,D\geq 2$ be two integers and $H^\sigma$ be a subdivision of a graph $H$ with maximum degree $D$ such that $v(H^\sigma)=n$. Also, let $\mathcal{H}$ be the hypergraph given in Lemma~\ref{lem:hypergraph} with parameters in Table~\ref{tbl:parameters}.
	Then, there exists a positive integer $n_0=n_0(k,D)$ such that for every $n\geq n_0$, the following holds.
	\begin{itemize}
		\item If $\sigma(e)$ is an even integer larger than $8\log_{D} (n)+\Omega (k^7\log^2(kD))$ for any $e\in E(H)$ and $F$ is an induced $k$-Ramsey graph for $C_6$ given in Lemma~\ref{lem:induced_even_odd_gadget}, then $\mathcal{H}(F)$ is induced $k$-Ramsey for $H^\sigma$. In particular, $\hat{R}_{\mathrm{ind}}(H^\sigma,k)=O(k^{342} \log^{9} (k) D^{9} \log D )n$.  
		\item If $\sigma(e)$ is an arbitrary integer larger than $12\log_{D} (n)+e^{\Omega(k\log k)}\log^2 D$ for any $e\in E(H)$ and $F$ is an induced $k$-Ramsey graph for $C_5$ given in Lemma~\ref{lem:induced_even_odd_gadget}, then $\mathcal{H}(F)$ is induced $k$-Ramsey for $H^\sigma$. In particular, $\hat{R}_{\mathrm{ind}}(H^\sigma,k)=e^{O(k\log k)} D^{9}\log (D)\, n$. 
	\end{itemize} 
\end{theorem}

For the non-induced size-Ramsey, these bounds can be substantially improved. 

\begin{theorem} \label{thm:noninduced_subd}
Let $k,D\geq 2$ be two integers and $H^\sigma$ be a subdivision of a graph $H$ with maximum degree $D$ such that $v(H^\sigma)=n$.  Also, let $\mathcal{H}'$ be the hypergraph given in Lemma~\ref{lem:hypergraph} with parameters in Table~\ref{tbl:parameters}.
	Then, there exists a positive integer $n_0=n_0(k,D)$ such that for every $n\geq n_0$, the following holds.
	\begin{itemize}
		\item If $\sigma(e)$ is an even integer larger than ${8\log_{D} (n)+\Omega(k^3 \log^2 D)}$ for any $e\in E(H)$ and $F$ is a $k$-Ramsey graph for $C_6$ given in Lemma~\ref{lem:even_gadget}, then $\mathcal{H}'(F)$ is $k$-Ramsey for $H^{\sigma}$. In particular, $\hat{R}(H^\sigma,k)={O(k^{34} D^5\log D)n}$.
				\item If $\sigma(e)$ is an arbitrary integer larger than ${2\log_2^2 (k)\log_{D} (n)+\Omega(k^4 2^{2k}\log^2 D)}$ for any $e\in E(H)$ and $F$ is the complete graph on $2^{2k}+1$ vertices, then $\mathcal{H}'(F)$ is $k$-Ramsey  for $H^{\sigma}$. In particular, $\hat{R}(H^\sigma,k)={O(2^{34k} k^6 \ln^5 k D^5\log D)n}$.
	\end{itemize}
\end{theorem}

We will prove Theorem~\ref{thm:induced_subd}. The proof of Theorem~\ref{thm:noninduced_subd} is very similar and we postpone it to the appendix. So, the rest of the paper is devoted to the proof of Theorem~\ref{thm:induced_subd}.

\section{Auxiliary graphs $G$ and $G'$}
Let  $\mathcal{H}$ be the hypergraph provided in Lemma~\ref{lem:hypergraph}. Also, let $\ell$ be equal to $6$ in the even case and equal to $5$ in the general case and $F$ be the graph which is induced $k$-Ramsey for $C_\ell$ (given in Lemma~\ref{lem:induced_even_odd_gadget}). 
In order to prove Theorem~\ref{thm:induced_subd}, as in~\cite{Sudakov-cycle}, we define
an auxiliary graph $G$ as follows.
The vertex set of $G$ is the same as the vertex set of $\mathcal{H}$. Fix an arbitrary $k$-edge coloring of $\mathcal{H}(F)$.  For each hyperedge of $\mathcal{H}$, in its corresponding copy of $F$, there exists a monochromatic induced copy of $C_\ell$. We pick one of these $\ell$-cycles and add an edge in $G$ between two vertices at distance two in the $\ell$-cycle. Also, we color this edge in $G$ by the same color as the $\ell$-cycle. So, we have a $k$-edge-colored graph $G$. Since $\mathcal{H}$ is linear, each edge $e\in E(G)$ is contained in a unique hyperedge of $\mathcal{H}$, which we denote by $h(e)$.

In the following lemma, we list some properties of the graph $G$.

\begin{lemma} \label{lem:subgraph}
The graph $G$ defined as above has the following properties. 
\begin{itemize}
\item[\rm (Q1)] $v(G)=v(\mathcal{H})=N$, $e(G)=e(\mathcal{H})\in [c N/2,c N]$, and $\Delta(G)\leq \Delta(\mathcal{H})\leq 8cs$.  
\item[\rm (Q2)]  For any $A \subseteq V({G})$ with $|A|\leq \alpha N$, we have $e_G(A)\leq (5/4)|A|$. 
\item[\rm (Q3)] $G$ contains a monochromatic subgraph $G'$ such that 
\begin{itemize}
	\item $v(G')\geq \alpha N$,
	\item the average degree of $G'$ is at least $c/(2k)=(10^7/2)\,Ds^4\ln s $, and
	\item $G'$ is a $\gamma$-expander graph, where $\gamma\geq {(10^4sk\log_2 (sD))^{-1}}$.
\end{itemize} 
\end{itemize}
\end{lemma}
In order to prove the lemma, we need the following result from \cite{Krivel:lower} (see also \cite{Sudakov-cycle}).
\begin{lemma} {\rm \cite{Krivel:lower,Sudakov-cycle}} \label{lem:krivel}
 Let $c_1>c_2>1$, $0<\alpha<1$, and $\Delta>0$. Let $G = (V,E)$ be a graph on $N$ vertices satisfying
 \begin{enumerate}
 	\item $|E|/|V| \geq c_1$,
 	\item for every subset $U\subseteq V$ of size $|U|\leq \alpha N$, we have $e_G(U)\leq c_2|U|$, and
 	\item $\Delta(G)\leq \Delta$.
 \end{enumerate}
Then $G$ contains an induced subgraph $G'=(V',E')$ on at least $\alpha N$ vertices with $|E'|/|V'|\geq (c_1+c_2)/2$ which is a $\gamma$-expander, where
$\gamma=(c_1-c_2)/(2\Delta\lceil \log_2(1/\alpha) \rceil)$.
\end{lemma}
\begin{proof}[Proof of Lemma~\ref{lem:subgraph}]
Property (Q1) immediately follows from the definition of $G$ and (Q2) follows from (P4).
Now, let $G_\mathrm{red}$ be the monochromatic subgraph of $G$ with the maximum number of edges. So, $e(G_\mathrm{red})\geq e(G)/k$. Apply Lemma~\ref{lem:krivel} with $G=G_\mathrm{red}$, $c_1=c/(2k)$, $c_2=5/4$ and let $G'$ be the induced subgraph of $G_\mathrm{red}$ guaranteed by the lemma. 

Since $c= 10^7Dks^4\ln s $, we have $\bar{d}(G')\geq c_1=  (10^7/2)Ds^4\ln s $. 
Finally, $G'$  a~$\gamma$-expander for
\begin{align*}
\gamma&\geq \dfrac{c_1 -5/4}{16\, cs\lceil \log_2 (1/\alpha)\rceil} \geq \dfrac{c }{64cks\log_2 (1/\alpha)} = \dfrac{1}{64\, sk\log_2(10^4c^5s^{15})}\geq\dfrac{1}{10^4sk\log_2(sD)}.
\end{align*}
\end{proof}

We also need a couple of definitions.
\begin{definition}
	For a subgraph $J$ of our auxiliary graph $G$, the \textit{closure} $\cl(J)$ of $J$  is the subset of $V(G)$ consisting of the union of all hyperedges $h(e)$ for all $e\in E(J)$.  Also, set $\cl_1(J)=\cl(J)$ and for each integer $k\geq 2$, the \textit{$k$th closure} $\cl_k(J)$ of $J$ is the subset of $V(G)$ consisting of the union of all hyperedges in $\mathcal{H}$ which intersect $\cl_{k-1}(J)$. 
\end{definition}

Here, we define the notion of an induced-good embedding in the graph $G$ with respect to the hypergraph $\mathcal{H}$.
\begin{definition}\label{def:good}
 A subgraph $J$ of $G$ is said to be an \textit{induced-good subgraph} of $G$ (with respect to $\mathcal{H}$) if 
 \begin{itemize}
 	\item for every $e_1,e_2\in E(J)$, if $e_1$ and $e_2$ are  not adjacent, then $h(e_1)\cap h(e_2)=\emptyset$, 
 	\item for every $h\in E(\mathcal{H})$, if $|h\cap \cl(J)|\ge 2$, then there exists an edge $e\in E(J)$ such that $h=h(e)$.
 \end{itemize}
 Also, an embedding $\phi:H\hookrightarrow G$ is called induced-good if $\phi(H)$ is an induced-good subgraph of~$G$.
 \end{definition}
Note that the second property of goodness implies that if $J$ is an induced-good subgraph of $G$, then $J$ is an induced subgraph of $G$.

Let $L_1=\lfloor (\ell-1)/2\rfloor $ and $L_2= \lceil (\ell+1)/2 \rceil  $. In order to find a monochromatic induced copy of $H^\sigma$ in $\mathcal{H}(F)$, the idea is to first find a monochromatic induced copy of $H^{\sigma'}$ as an induced-good subgraph of $G$, for some $\sigma'$. The following lemma shows how we can convert an $H^{\sigma'}$ in $G$ to an $H^\sigma$ in $\mathcal{H}(F)$.

\begin{lemma} \label{lem:sigma'tosigma}
Suppose that $H^{\sigma'}$ is a monochromatic induced-good subgraph of $G$ such that   $\sigma(e)/L_2\leq \sigma'(e)\leq \sigma(e)/L_1$, for all $e\in E(H)$. Then, $\mathcal{H}(F)$ contains a monochromatic copy of $H^\sigma$ as an induced subgraph.
\end{lemma}
\begin{proof}
By definition of $G$, for each edge $e$ in $E(G)$, there is an $\ell$-cycle in $\mathcal{H}(F)$ of the same color whose vertices are  in $h(e)$ and the endpoints of $e$ are at distance two in the $\ell$-cycle.
 Therefore, we can find two induced paths in $\mathcal{H}(F)$ of lengths $L_1$ and $L_2$ between the endpoints of~$e$.
 
Now, suppose that $H^{\sigma'}$ is a monochromatic induced-good subgraph of $G$ and let $uv\in E(H)$ and $P$ be a monochromatic $uv$-path of length $\sigma'(uv)$ in $H^{\sigma'}$. So, for each edge $e\in E(P)$, we can find two paths in  $\mathcal{H}(F)$ of lengths $L_1$ and $L_2$ between the endpoints of $e$ with the same color as $e$. 
Because of the first property of goodness and linearity of $\mathcal{H}$, these paths do not intersect each other except for their endpoints. Now, for each edge $e\in E(P)$, we select one of the paths of lengths $L_1$ or $L_2$ so that the union of these paths forms a path of length $\sigma(e)$ in $\mathcal{H}(F)$ between $u$ and $v$. 

Therefore, $\mathcal{H}(F)$ contains $H^\sigma$ as a monochromatic subgraph. We just need to prove that $H^\sigma$ is an induced subgraph. Suppose that $e=xy$ is an edge in $E(\mathcal{H}(F))\setminus E(H^\sigma)$ where $x,y\in V(H^\sigma)$ and let $x,y$ be in a hyperedge  $h\in E(\mathcal{H})$. Then, $h$ intersects $\cl(H^{\sigma'})$ in at least two vertices $x,y$. Thus, by the second property of goodness, $h=h(e')$ for some $e'\in E(H^{\sigma'})$. Also, let $x\in h(e_1)$ and $y\in h(e_2)$ for some $e_1,e_2\in E(H^{\sigma'})$.  So, $h(e')$ intersects $h(e_1)$ and $h(e_2)$ in $x$ and $y$, respectively. Consequently, by the first property of goodness, $e'$ is adjacent to $e_1$ and $e_2$ in $G$ and by linearity of $\mathcal{H}$, $e'\cap e_1=\{x\}$ and $e'\cap e_2=\{y\}$. Hence, $e'=xy$. This is impossible because $E(G)$ and $E(\mathcal{H}(F))$ are disjoint.  This contradiction implies that $H^\sigma$ is an induced subgraph of $\mathcal{H}(F)$. 
\end{proof}

\section{Critical sets}
Let $G'$ be the graph defined in Lemma~\ref{lem:subgraph}. 
In a good part of our embedding process, we shall grow induced-good subgraphs of maximum degree $D$ in $G'$ by adding pendant edges. However, during the embedding, some vertices appear whose  neighbors are not mostly usable for extending the subgraph (since either these neighbors are inside the already embedded graph, or adding them violates goodness conditions). We call these vertices critical vertices. So, we need to simultaneously reserve some neighbors for critical vertices. Let us define this notion more precisely. 
\begin{definition}\label{def:induced-available}
Let $J$ be a subgraph of $G'$. For a vertex $v\in V(G')$, a neighbor $u$ in $N_{G'}(v)$  is called an \textit{induced-available} neighbor of $v$ (with respect to $J$) if  $h(uv)\setminus\{v\}$ does not intersect $\cl_2(J)$. The set of all induced-available neighbors of $v$ is denoted by $A^{\mathrm{ind}}_J(v)$.   
\end{definition}

\begin{definition}\label{def:induced-critical}
Let $d,d' \ge 2$ be integers and let $J$ be a subgraph of $G'$.
We define the set $C^{\mathrm{ind}}_{d,d'}(J)$ of \emph{$(d,d')$-critical vertices} of $J$ by the following iterative procedure.

\smallskip
\noindent
\textbf{Step 0.}  
Set $J_0 := J$ and let $X_0$ be the set of all non-leaf vertices of $J_0$.

\noindent
\textbf{Step $i\ge 1$.}  
Given $J_{i-1}$ and $X_0,\ldots,X_{i-1}$. Define 
\[
X_i=\{
v\in V(G') \setminus \bigcup_{j =0}^{i-1} X_j\, : \, |N_{G'}(v)\setminus A^{\mathrm{ind}}_{J_{i-1}}|\geq d \}.
\]
Indeed, $X_i$ is the set of all vertices in $V(G') \setminus \bigcup_{j =0}^{i-1} X_j$
which have at least $d$ neighbours in $G'$ that are not induced-available with respect to $J_{i-1}$ (see Figure~\ref{fig:critical}).
For every vertex $v \in X_i$, choose arbitrarily $d'$ edges of $G'$ incident with $v$ (if $\deg_{G'}(v) < d'$, then no edge is added) and let $J_i$ be the graph obtained from $J_{i-1}$ by adding these selected edges over all vertices of $X_i$.
Continue the iteration until some $X_i$ is empty.  
We then define
\[
C^{\mathrm{ind}}_{d,d'}(J) := \bigcup_{i \ge 0} X_i .
\]
	\end{definition}
     It should be noted that the $(d,d')$-critical set of $J$ is not unique, because in each step the $d'$ incident edges to new vertices are chosen arbitrarily.
     \begin{figure}
\begin{center}
    \tikzset{every picture/.style={line width=0.75pt}} %

\begin{tikzpicture}[x=0.75pt,y=0.75pt,yscale=-1,xscale=1]
	
	\draw  [fill={rgb, 255:red, 5; green, 5; blue, 5 }  ,fill opacity=1 ] (248.29,66.5) .. controls (248.29,62.91) and (251.13,60) .. (254.63,60) .. controls (258.13,60) and (260.96,62.91) .. (260.96,66.5) .. controls (260.96,70.09) and (258.13,73) .. (254.63,73) .. controls (251.13,73) and (248.29,70.09) .. (248.29,66.5) -- cycle ;
	\draw  [fill={rgb, 255:red, 5; green, 5; blue, 5 }  ,fill opacity=1 ] (248.29,173.5) .. controls (248.29,169.91) and (251.13,167) .. (254.63,167) .. controls (258.13,167) and (260.96,169.91) .. (260.96,173.5) .. controls (260.96,177.09) and (258.13,180) .. (254.63,180) .. controls (251.13,180) and (248.29,177.09) .. (248.29,173.5) -- cycle ;
	\draw  [fill={rgb, 255:red, 5; green, 5; blue, 5 }  ,fill opacity=1 ] (248.29,116.5) .. controls (248.29,112.91) and (251.13,110) .. (254.63,110) .. controls (258.13,110) and (260.96,112.91) .. (260.96,116.5) .. controls (260.96,120.09) and (258.13,123) .. (254.63,123) .. controls (251.13,123) and (248.29,120.09) .. (248.29,116.5) -- cycle ;
	\draw  [fill={rgb, 255:red, 5; green, 5; blue, 5 }  ,fill opacity=1 ] (188.29,173.5) .. controls (188.29,169.91) and (191.13,167) .. (194.63,167) .. controls (198.13,167) and (200.96,169.91) .. (200.96,173.5) .. controls (200.96,177.09) and (198.13,180) .. (194.63,180) .. controls (191.13,180) and (188.29,177.09) .. (188.29,173.5) -- cycle ;
	\draw [line width=2.25]    (254.63,66.5) -- (254.63,173.5) ;
	\draw [line width=2.25]    (194.63,173.5) -- (254.63,173.5) ;
	\draw [line width=2.25]    (254.63,66.5) -- (194.63,173.5) ;
	\draw  [fill={rgb, 255:red, 5; green, 5; blue, 5 }  ,fill opacity=1 ] (151.95,200) .. controls (151.95,196.41) and (154.79,193.5) .. (158.29,193.5) .. controls (161.79,193.5) and (164.63,196.41) .. (164.63,200) .. controls (164.63,203.59) and (161.79,206.5) .. (158.29,206.5) .. controls (154.79,206.5) and (151.95,203.59) .. (151.95,200) -- cycle ;
	\draw [line width=2.25]    (158.29,200) -- (194.63,173.5) ;
	\draw  [fill={rgb, 255:red, 5; green, 5; blue, 5 }  ,fill opacity=1 ] (373.61,116.5) .. controls (373.61,112.91) and (376.45,110) .. (379.95,110) .. controls (383.45,110) and (386.29,112.91) .. (386.29,116.5) .. controls (386.29,120.09) and (383.45,123) .. (379.95,123) .. controls (376.45,123) and (373.61,120.09) .. (373.61,116.5) -- cycle ;
	\draw  [fill={rgb, 255:red, 5; green, 5; blue, 5 }  ,fill opacity=1 ] (315.61,116.5) .. controls (315.61,112.91) and (318.45,110) .. (321.95,110) .. controls (325.45,110) and (328.29,112.91) .. (328.29,116.5) .. controls (328.29,120.09) and (325.45,123) .. (321.95,123) .. controls (318.45,123) and (315.61,120.09) .. (315.61,116.5) -- cycle ;
	\draw    (254.63,66.5) -- (379.95,116.5) ;
	\draw    (379.95,116.5) -- (321.95,116.5) ;
	\draw  [fill={rgb, 255:red, 5; green, 5; blue, 5 }  ,fill opacity=1 ] (312.99,170.54) .. controls (312.99,166.95) and (315.82,164.04) .. (319.32,164.04) .. controls (322.82,164.04) and (325.66,166.95) .. (325.66,170.54) .. controls (325.66,174.13) and (322.82,177.04) .. (319.32,177.04) .. controls (315.82,177.04) and (312.99,174.13) .. (312.99,170.54) -- cycle ;
	\draw    (319.32,170.54) -- (379.95,116.5) ;
	\draw  [dash pattern={on 4.5pt off 4.5pt}] (232.29,117) .. controls (232.29,104.3) and (269.45,94) .. (315.29,94) .. controls (361.13,94) and (398.29,104.3) .. (398.29,117) .. controls (398.29,129.7) and (361.13,140) .. (315.29,140) .. controls (269.45,140) and (232.29,129.7) .. (232.29,117) -- cycle ;
	\draw  [dash pattern={on 4.5pt off 4.5pt}] (259.91,239.68) .. controls (250.63,230.86) and (272.54,192.75) .. (308.86,154.56) .. controls (345.17,116.37) and (382.13,92.57) .. (391.41,101.39) .. controls (400.69,110.22) and (378.78,148.33) .. (342.46,186.52) .. controls (306.15,224.71) and (269.19,248.51) .. (259.91,239.68) -- cycle ;
	\draw  [dash pattern={on 4.5pt off 4.5pt}] (212.49,216.86) .. controls (184.87,193.07) and (172.81,161.8) .. (185.54,147.03) .. controls (198.27,132.25) and (230.98,139.56) .. (258.59,163.35) .. controls (286.21,187.15) and (298.27,218.42) .. (285.54,233.19) .. controls (272.81,247.97) and (240.1,240.66) .. (212.49,216.86) -- cycle ;
	\draw  [fill={rgb, 255:red, 5; green, 5; blue, 5 }  ,fill opacity=1 ] (265.61,223.5) .. controls (265.61,219.91) and (268.45,217) .. (271.95,217) .. controls (275.45,217) and (278.29,219.91) .. (278.29,223.5) .. controls (278.29,227.09) and (275.45,230) .. (271.95,230) .. controls (268.45,230) and (265.61,227.09) .. (265.61,223.5) -- cycle ;
	\draw  [dash pattern={on 4.5pt off 4.5pt}] (164.61,159.59) .. controls (184.51,149.03) and (217.28,171.84) .. (237.79,210.52) .. controls (258.31,249.2) and (258.81,289.12) .. (238.9,299.67) .. controls (219,310.23) and (186.23,287.43) .. (165.72,248.75) .. controls (145.2,210.06) and (144.71,170.15) .. (164.61,159.59) -- cycle ;
	\draw  [fill={rgb, 255:red, 5; green, 5; blue, 5 }  ,fill opacity=1 ] (224.29,286.5) .. controls (224.29,282.91) and (227.13,280) .. (230.63,280) .. controls (234.13,280) and (236.96,282.91) .. (236.96,286.5) .. controls (236.96,290.09) and (234.13,293) .. (230.63,293) .. controls (227.13,293) and (224.29,290.09) .. (224.29,286.5) -- cycle ;
	\draw  [dash pattern={on 4.5pt off 4.5pt}] (206.45,287.74) .. controls (206.6,276.69) and (238.47,268.17) .. (277.64,268.71) .. controls (316.81,269.24) and (348.44,278.63) .. (348.29,289.67) .. controls (348.14,300.71) and (316.27,309.24) .. (277.1,308.7) .. controls (237.93,308.17) and (206.3,298.78) .. (206.45,287.74) -- cycle ;
	\draw  [fill={rgb, 255:red, 5; green, 5; blue, 5 }  ,fill opacity=1 ] (325.61,286.5) .. controls (325.61,282.91) and (328.45,280) .. (331.95,280) .. controls (335.45,280) and (338.29,282.91) .. (338.29,286.5) .. controls (338.29,290.09) and (335.45,293) .. (331.95,293) .. controls (328.45,293) and (325.61,290.09) .. (325.61,286.5) -- cycle ;
	\draw  [fill={rgb, 255:red, 5; green, 5; blue, 5 }  ,fill opacity=1 ] (351.95,216.5) .. controls (351.95,212.91) and (354.79,210) .. (358.29,210) .. controls (361.79,210) and (364.63,212.91) .. (364.63,216.5) .. controls (364.63,220.09) and (361.79,223) .. (358.29,223) .. controls (354.79,223) and (351.95,220.09) .. (351.95,216.5) -- cycle ;
	\draw    (358.29,216.5) -- (379.95,116.5) ;
	\draw  [dash pattern={on 4.5pt off 4.5pt}] (331.98,301.05) .. controls (316.9,296.63) and (317.96,247.67) .. (334.36,191.68) .. controls (350.75,135.69) and (376.26,93.89) .. (391.35,98.3) .. controls (406.43,102.72) and (405.36,151.68) .. (388.97,207.67) .. controls (372.58,263.66) and (347.06,305.47) .. (331.98,301.05) -- cycle ;
	
	\draw (199.29,92) node  {$J$};
	\draw (375,92) node  {$v$};
	\draw (260,50) node {$v_1$};
	\draw (320,103) node {$v_2$};
	\draw (320,158) node {$v_3$};
	\draw (350,201) node {$v_d$};	
\end{tikzpicture}

\end{center}
    \caption{Critical Vertices.  The neighbors $v_1$, $v_2$ of $v$ are not induced-available since $h(vv_1)\setminus \{v\}$ and $h(vv_2)\setminus \{v\}$ intersect $V(J)$. Also, $v_3$ and $v_d$ are not induced-available because $h(vv_3)\setminus \{v\}$ intersects $\cl(J)$ and $h(vv_d)\setminus \{v\}$ intersects $\cl_2(J)$.}\label{fig:critical} 
\end{figure}
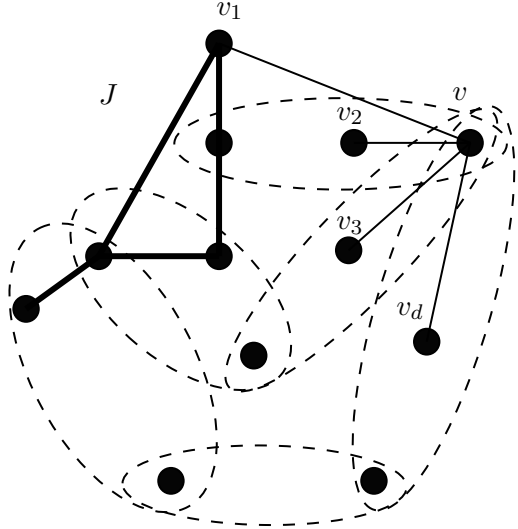

As in~\cite{Girao}, it is crucial that the set of critical vertices
does not explode as we proceed.  We shall confirm this fact in our
next lemma, the proof of which is based on the sparsity property~(P4)
of the hypergraph~$\mathcal{H}$.  A lemma similar to
Lemma~\ref{lem:induced-critical} below appears in~\cite{Girao}.

\begin{lemma} \label{lem:induced-critical}
	Let $d,d'\geq 2$ be two positive integers such that $d=  20\,sd'$. 
	Also, let $J$ be a subgraph of $G'$ where  $e(J)\leq \alpha N/(22s)-d'$. Also, let $X_0$ be the set of all non-leaves of $J$.  Then, $|C^\mathrm{ind}_{d,d'}(J)|< |X_0|+e(J)/d'$. 
\end{lemma}

\begin{proof}
	Suppose that $x_1,\ldots, x_{t}$ are the first $t$ vertices outside $X_0$ which are added to the critical set of $J$ and for the contrary suppose that $t= \lceil e(J)/d' \rceil$. 
	Let $\hat{J}$ be the graph obtained from $J$ by adding $d'$ arbitrary neighbors for each $x_i$, $i\in [t]$. By the definition, each $x_i$ has $d$ many neighbors $x_{i_1},\ldots,x_{i_d}$ such that for each $j\in [d]$, $h(x_i x_{i_j})\setminus\{x_i\}$ intersects $\cl_2(\hat{J})$. Define two subsets $E_1,E_2\subseteq E(\mathcal{H})$ and a subset $Y\subseteq V(\mathcal{H})$ as follows. Let $E_1$ be the set of all hyperedges $h(x_i x_{i_j})$ for $i\in [t]$ and $j\in [d]$. For each $h(e)\in E_1$, if $h(e)$ does not intersect $\cl(\hat{J})$, then there is a hyperedge $h$ which intersects $h(e)$ and $\cl(\hat{J})$. Add this hyperedge $h$ to the set $E_2$ and add a vertex in $h\cap h(e)$ to the set $Y$. Also, set $A=\cl(\hat{J}) \cup Y$. Every hyperedge $h(e)\in E_1$ intersects $A$ in at least two vertices. Also, 
	\[\sum_{h\in E_2} |h\cap Y|\geq |Y|. \]
	Therefore, we have
	\begin{equation}\label{eq:cap}
	\sum_{h: |h\cap A|\geq 2} |h\cap A|\geq 2|E_1|+|Y|.
	\end{equation}
	On the other hand, 
	
	\begin{align*}
	|A|=|\cl(\hat{J})|+|Y|\leq e(J)s+t d' s+ t d = e(J)s+21td's\leq 22e(J)s+21d's\leq \alpha N.
	\end{align*}
	Then, by (P4), we have
	\begin{align*}
	\sum_{h: |h\cap A|\geq 2} |h\cap A|< \dfrac{5}{2} |A| = \dfrac{5}{2} (|\cl(\hat{J})|+|Y|) \leq \dfrac{5}{2} (e(J)s+t d' s+|Y|) \leq 5\,t d's +\dfrac{5}{2} |Y|. 	
	\end{align*}
	Now, combining with \eqref{eq:cap} we have
	\[
	\dfrac{1}{2}|E_1|\leq 2|E_1|-\dfrac{3}{2} |Y| < 5 t d' s,
	\]
where the first inequality holds because $|Y|\leq |E_1|$. On the other hand, $|E_1|\geq t d/2$. Thus, $d< 20\, d's$, a contradiction. 
\end{proof}

\section{Extending good embedding }
\label{sec:extending}
In this section, we develop a machinery for extending an induced-good
embedding by adding trees pending from its vertices (See
Figure~\ref{fig:extension}).  The following lemma, which is heavily
inspired in~\cite[Lemma~2.13]{Girao}, is a key component in this machinery.
The graph $G'$ is the graph given in Lemma~\ref{lem:subgraph}.

\begin{lemma} \label{lem:extend-induced} Let $s\geq 3$ and $D\geq 2$
  be integers and $K$ be an induced-good subgraph of $G'$. Also, let
  $C\subseteq V(G')$ be a subset of vertices such that
  $|C|\leq \alpha N/(10^5Ds^3\ln s)$. If for each vertex $v\in C$ we
  have $|A^{\mathrm{ind}}_K(v)|\geq 10^5\, D s^3\ln s $, then $K$ can
  be extended to a subgraph $\widehat{K}$, where $\widehat{K}$ is
  obtained from $K$ by adding some disjoint rooted trees
  $T_1,\ldots, T_c$, where $c=|V(K)\cap C|$ with roots in $V(K)\cap C$
  such that
  \begin{itemize}
  \item $\widehat{K}$ is an induced-good subgraph of $G'$.
  \item Every non-leaf vertex of the $T_i$ is in $C$ and has $D$ children.
  \item No leaf of the $T_i$ is in $C$.  
  \end{itemize} 
\end{lemma}
\begin{proof}
	If $V(K)\cap C=\emptyset$, then we can take $\widehat{K}=K$ and all conditions hold. Now, suppose that $V(K)\cap C\neq\emptyset$. For each vertex $v\in C$, choose $ 10^5D s^3 \ln s $ induced-available neighbors with respect to $K$ and let $K^*$ be the graph of union of all these edges. 
	
	Now, we define an auxiliary bipartite hypergraph $\tilde{\mathcal{H}}$ as follows. First, let $C'$ be a disjoint copy of $C$ of new vertices and let $V(\tilde{\mathcal{H}})= C'\cup \cl(K^*)$. Also, for every edge $e=xy\in E(K^*)$, if $x\in C$, we put the hyperedge $h(e)\setminus \{x\}\cup \{x'\}$ in $\tilde{\mathcal{H}}$, where $x'$ is the copy of $x$ in $C'$. So, if $x,y\in C$, then there are two hyperedges in $\tilde{\mathcal{H}}$, namely $h(e)\setminus \{x\}\cup \{x'\}$ and $h(e)\setminus \{y\}\cup \{y'\}$. The hypergraph $\tilde{\mathcal{H}}$ is bipartite, since every hyperedge intersects $C'$ in one vertex. Also, every vertex $x'\in C'$ is contained in at least $ 10^5Ds^3\ln s $ hyperedges. Let $D'=10^4Ds^2\ln s $. Now, we apply Corollary~\ref{cor:aharoni} for the bipartite hypergraph $\tilde{\mathcal{H}}$. 
	First, suppose that for every set $I'\subseteq C'$, we have
	\begin{equation} \label{eq:aharoni3}
	\tau(N_{\tilde{\mathcal{H}}}(I'))\geq (2s-3) (D'|I'|-1)+1.
	\end{equation}
	(We will later prove that the inequality holds.) Then, by Corollary~\ref{cor:aharoni}, $\tilde{\mathcal{H}}$ admits a $D'$-matching, i.e. for each vertex $x'\in C'$, we can choose $D'$ hyperedges $h^{x'}_1,\ldots, h^{x'}_{D'}$ incident with $x'$ such that for every $x',y'\in C'$ and every $i,j\in [D']$, the sets $h^{x'}_i\setminus \{x'\}$ and $h^{y'}_j\setminus \{y'\}$ are disjoint.
	
	Now, for every $x\in C$, let $A_x$ be the set of $D'$ neighbors of $x$ in $h_1^{x'}, \ldots, h_{D'}^{x'}$.  
	We define an auxiliary graph $L$ with the vertex set $\cup_{x\in C}A_x$ where two arbitrary vertices $u\in A_x$ and $v\in A_y$ are adjacent if there exists a hyperedge $ h\in E(\mathcal{H}) $ ($h\neq h(ux), h(yv)$) intersecting both sets $h(xu)\setminus\{x\}$ and $h(yv)\setminus\{y\}$. The graph $L$ satisfies the following properties. (We will check these properties at the end of the proof.)
	\begin{itemize}
		\item $e_{L}(A_x, A_y)\leq 1$ for every two vertices $x, y\in C$.
		\item $L$ is $3s^2$-degenerate.
		\item The maximum degree of $L$, $\Delta_L$, is at most $(10s)^8D$. 
 	\end{itemize} 
	
	We are going to apply Lemma~\ref{lem:LLL} by setting $G=L$, $\Delta_L\leq (10s)^8D$, $d=3s^2$, $a=D'$ and $t=|C|$. So, for each $x\in C$, we can find a subset $A'_x\subseteq A_x$, such that $|A'_x|\geq D'/(300s^2)> D$ and $ L$ induces an stable set on $\cup_{x\in C} A'_x$. (Note that we have $a=D'\geq 600\, s^2 \ln \Delta_L $ and so the last condition of Lemma~\ref{lem:LLL} is satisfied.) 
	So, for every $x,y\in C$ and every $u\in A'_x$ and $v\in A'_y$, we have $h(xu)\setminus \{x\}$ and $h(yv)\setminus \{y\}$ are disjoint and there is no hyperedge $h\in E(\mathcal{H})$, where $h$ intersects both $h(xu)\setminus \{x\}$ and $h(yv)\setminus \{y\}$.
	Now, we construct a forest $\tilde{K}$ as follows. 
	
	First, we pick a vertex $x\in C\cap V(K)$. The vertex $x$ has $D$ neighbors in $A'_x$. We add all these neighbors to $\tilde{K}$ and orient all edges from $x$ to its neighbors. If these neighbors of $x$ are outside of $C$, then we stop here and go to the next vertex in $C\cap V(K) $. If one of these neighbors, say $y$, is in $C$, then we continue from $y$, so $y$ has $D$ neighbors in $A'_y$. We continue this process by adding these neighbors to $\tilde{K}$ and orienting the edges from $y$ until all leaves are outside of $C$. Let us call the obtained graph $\tilde{K}$. We have to show that $\tilde{K}$ is a forest. First, note that every vertex has in-degree at most one. Otherwise, if a vertex $x$ has two in-neighbors $y_1,y_2$, then $x\in A'_{y_1} \cap A'_{y_2}$. This contradicts the fact that these sets are disjoint.
	We also need to show that if $x,y\in C$ are adjacent, then $xy$ cannot be oriented in both directions. To see this, suppose that $xy$ is oriented in both directions, so $y\in A'_x$ and $x\in A'_y$ and thus $h(xy)\setminus \{x\} $ and $h(xy)\setminus \{y\}$ are disjoint. This contradicts the fact that  $s\geq 3$. 
	
	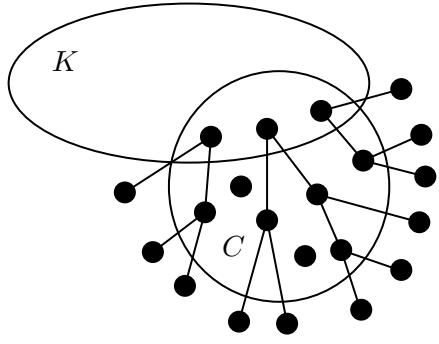
\begin{figure}
	\begin{center}
	    \tikzset{every picture/.style={line width=0.75pt}} %

\begin{tikzpicture}[x=0.75pt,y=0.75pt,yscale=-1,xscale=1]
	
	\draw   (40,70) .. controls (40,47.91) and (80.29,30) .. (130,30) .. controls (179.71,30) and (220,47.91) .. (220,70) .. controls (220,92.09) and (179.71,110) .. (130,110) .. controls (80.29,110) and (40,92.09) .. (40,70) -- cycle ;
	\draw   (120,122) .. controls (120,89.97) and (144.62,64) .. (175,64) .. controls (205.38,64) and (230,89.97) .. (230,122) .. controls (230,154.03) and (205.38,180) .. (175,180) .. controls (144.62,180) and (120,154.03) .. (120,122) -- cycle ;
	\draw  [fill={rgb, 255:red, 0; green, 0; blue, 0 }  ,fill opacity=1 ] (136,97) .. controls (136,94.24) and (138.24,92) .. (141,92) .. controls (143.76,92) and (146,94.24) .. (146,97) .. controls (146,99.76) and (143.76,102) .. (141,102) .. controls (138.24,102) and (136,99.76) .. (136,97) -- cycle ;
	\draw  [fill={rgb, 255:red, 0; green, 0; blue, 0 }  ,fill opacity=1 ] (164,93) .. controls (164,90.24) and (166.24,88) .. (169,88) .. controls (171.76,88) and (174,90.24) .. (174,93) .. controls (174,95.76) and (171.76,98) .. (169,98) .. controls (166.24,98) and (164,95.76) .. (164,93) -- cycle ;
	\draw  [fill={rgb, 255:red, 0; green, 0; blue, 0 }  ,fill opacity=1 ] (191,84) .. controls (191,81.24) and (193.24,79) .. (196,79) .. controls (198.76,79) and (201,81.24) .. (201,84) .. controls (201,86.76) and (198.76,89) .. (196,89) .. controls (193.24,89) and (191,86.76) .. (191,84) -- cycle ;
	\draw  [fill={rgb, 255:red, 0; green, 0; blue, 0 }  ,fill opacity=1 ] (93,125) .. controls (93,122.24) and (95.24,120) .. (98,120) .. controls (100.76,120) and (103,122.24) .. (103,125) .. controls (103,127.76) and (100.76,130) .. (98,130) .. controls (95.24,130) and (93,127.76) .. (93,125) -- cycle ;
	\draw  [fill={rgb, 255:red, 0; green, 0; blue, 0 }  ,fill opacity=1 ] (133,135) .. controls (133,132.24) and (135.24,130) .. (138,130) .. controls (140.76,130) and (143,132.24) .. (143,135) .. controls (143,137.76) and (140.76,140) .. (138,140) .. controls (135.24,140) and (133,137.76) .. (133,135) -- cycle ;
	\draw  [fill={rgb, 255:red, 0; green, 0; blue, 0 }  ,fill opacity=1 ] (164,139) .. controls (164,136.24) and (166.24,134) .. (169,134) .. controls (171.76,134) and (174,136.24) .. (174,139) .. controls (174,141.76) and (171.76,144) .. (169,144) .. controls (166.24,144) and (164,141.76) .. (164,139) -- cycle ;
	\draw  [fill={rgb, 255:red, 0; green, 0; blue, 0 }  ,fill opacity=1 ] (189,126) .. controls (189,123.24) and (191.24,121) .. (194,121) .. controls (196.76,121) and (199,123.24) .. (199,126) .. controls (199,128.76) and (196.76,131) .. (194,131) .. controls (191.24,131) and (189,128.76) .. (189,126) -- cycle ;
	\draw  [fill={rgb, 255:red, 0; green, 0; blue, 0 }  ,fill opacity=1 ] (212,109) .. controls (212,106.24) and (214.24,104) .. (217,104) .. controls (219.76,104) and (222,106.24) .. (222,109) .. controls (222,111.76) and (219.76,114) .. (217,114) .. controls (214.24,114) and (212,111.76) .. (212,109) -- cycle ;
	\draw  [fill={rgb, 255:red, 0; green, 0; blue, 0 }  ,fill opacity=1 ] (231,73) .. controls (231,70.24) and (233.24,68) .. (236,68) .. controls (238.76,68) and (241,70.24) .. (241,73) .. controls (241,75.76) and (238.76,78) .. (236,78) .. controls (233.24,78) and (231,75.76) .. (231,73) -- cycle ;
	\draw  [fill={rgb, 255:red, 0; green, 0; blue, 0 }  ,fill opacity=1 ] (201,154) .. controls (201,151.24) and (203.24,149) .. (206,149) .. controls (208.76,149) and (211,151.24) .. (211,154) .. controls (211,156.76) and (208.76,159) .. (206,159) .. controls (203.24,159) and (201,156.76) .. (201,154) -- cycle ;
	\draw  [fill={rgb, 255:red, 0; green, 0; blue, 0 }  ,fill opacity=1 ] (107,155) .. controls (107,152.24) and (109.24,150) .. (112,150) .. controls (114.76,150) and (117,152.24) .. (117,155) .. controls (117,157.76) and (114.76,160) .. (112,160) .. controls (109.24,160) and (107,157.76) .. (107,155) -- cycle ;
	\draw  [fill={rgb, 255:red, 0; green, 0; blue, 0 }  ,fill opacity=1 ] (123,172) .. controls (123,169.24) and (125.24,167) .. (128,167) .. controls (130.76,167) and (133,169.24) .. (133,172) .. controls (133,174.76) and (130.76,177) .. (128,177) .. controls (125.24,177) and (123,174.76) .. (123,172) -- cycle ;
	\draw  [fill={rgb, 255:red, 0; green, 0; blue, 0 }  ,fill opacity=1 ] (150,190) .. controls (150,187.24) and (152.24,185) .. (155,185) .. controls (157.76,185) and (160,187.24) .. (160,190) .. controls (160,192.76) and (157.76,195) .. (155,195) .. controls (152.24,195) and (150,192.76) .. (150,190) -- cycle ;
	\draw  [fill={rgb, 255:red, 0; green, 0; blue, 0 }  ,fill opacity=1 ] (174,191) .. controls (174,188.24) and (176.24,186) .. (179,186) .. controls (181.76,186) and (184,188.24) .. (184,191) .. controls (184,193.76) and (181.76,196) .. (179,196) .. controls (176.24,196) and (174,193.76) .. (174,191) -- cycle ;
	\draw  [fill={rgb, 255:red, 0; green, 0; blue, 0 }  ,fill opacity=1 ] (211,184) .. controls (211,181.24) and (213.24,179) .. (216,179) .. controls (218.76,179) and (221,181.24) .. (221,184) .. controls (221,186.76) and (218.76,189) .. (216,189) .. controls (213.24,189) and (211,186.76) .. (211,184) -- cycle ;
	\draw  [fill={rgb, 255:red, 0; green, 0; blue, 0 }  ,fill opacity=1 ] (231,164) .. controls (231,161.24) and (233.24,159) .. (236,159) .. controls (238.76,159) and (241,161.24) .. (241,164) .. controls (241,166.76) and (238.76,169) .. (236,169) .. controls (233.24,169) and (231,166.76) .. (231,164) -- cycle ;
	\draw  [fill={rgb, 255:red, 0; green, 0; blue, 0 }  ,fill opacity=1 ] (240,140) .. controls (240,137.24) and (242.24,135) .. (245,135) .. controls (247.76,135) and (250,137.24) .. (250,140) .. controls (250,142.76) and (247.76,145) .. (245,145) .. controls (242.24,145) and (240,142.76) .. (240,140) -- cycle ;
	\draw  [fill={rgb, 255:red, 0; green, 0; blue, 0 }  ,fill opacity=1 ] (241,96) .. controls (241,93.24) and (243.24,91) .. (246,91) .. controls (248.76,91) and (251,93.24) .. (251,96) .. controls (251,98.76) and (248.76,101) .. (246,101) .. controls (243.24,101) and (241,98.76) .. (241,96) -- cycle ;
	\draw  [fill={rgb, 255:red, 0; green, 0; blue, 0 }  ,fill opacity=1 ] (243,117) .. controls (243,114.24) and (245.24,112) .. (248,112) .. controls (250.76,112) and (253,114.24) .. (253,117) .. controls (253,119.76) and (250.76,122) .. (248,122) .. controls (245.24,122) and (243,119.76) .. (243,117) -- cycle ;
	\draw    (141,97) -- (98,125) ;
	\draw    (141,97) -- (138,135) ;
	\draw    (169,93) -- (169,139) ;
	\draw    (169,93) -- (194,126) ;
	\draw    (196,84) -- (236,73) ;
	\draw    (196,84) -- (217,109) ;
	\draw    (217,109) -- (246,96) ;
	\draw    (217,109) -- (248,117) ;
	\draw    (138,135) -- (112,155) ;
	\draw    (138,135) -- (128,172) ;
	\draw    (169,139) -- (155,190) ;
	\draw    (169,139) -- (179,191) ;
	\draw    (194,126) -- (206,154) ;
	\draw    (194,126) -- (245,140) ;
	\draw    (206,154) -- (214,179) ;
	\draw    (206,154) -- (232,163) ;
	\draw  [fill={rgb, 255:red, 0; green, 0; blue, 0 }  ,fill opacity=1 ] (183,157) .. controls (183,154.24) and (185.24,152) .. (188,152) .. controls (190.76,152) and (193,154.24) .. (193,157) .. controls (193,159.76) and (190.76,162) .. (188,162) .. controls (185.24,162) and (183,159.76) .. (183,157) -- cycle ;
	\draw  [fill={rgb, 255:red, 0; green, 0; blue, 0 }  ,fill opacity=1 ] (151,122) .. controls (151,119.24) and (153.24,117) .. (156,117) .. controls (158.76,117) and (161,119.24) .. (161,122) .. controls (161,124.76) and (158.76,127) .. (156,127) .. controls (153.24,127) and (151,124.76) .. (151,122) -- cycle ;
	
	\draw (60,52.4) node [anchor=north west][inner sep=0.75pt]    {$K$};
	\draw (145,145.4) node [anchor=north west][inner sep=0.75pt]    {$C$};

\end{tikzpicture}

	\end{center}
            \caption{Extension of an induced-good embedding.}\label{fig:extension} 
	\end{figure}

	Hence, every connected component of $\tilde{K}$ is unicyclic. So, the only case that $\tilde{K}$ is not a forest, is that when a vertex $x\in C\cap V(K)$ has in-degree one. However, this is impossible, because the out-neighbors of each vertex in $\tilde{K}$ are induced-available with respect to $K$ and so are not in $V(K)$. Hence, $\tilde{K}$ is a forest. Now, we define $\widehat{K}$ to be the union of $K$ and $\tilde{K}$. By the construction, it is clear that leaves of $\tilde{K}$ are not in $C$ and every non-leaf vertex of $\tilde{K}$ is in $C$ and has $D$ children. So, it remains to prove that $\widehat{K}$ is induced-good. To see this, first let $e_1,e_2\in E(\widehat{K})$ be two non-adjacent edges. If either $e_1$ or $e_2$ are in $E(K)$, then $h(e_1)\cap h(e_2)=\emptyset$, because $K$ is induced-good and the edges of $\tilde{K}$ are induced-available with respect to $K$. Also, suppose that $e_1=xu$ and $e_2=yv$ are in $E(\tilde{K})$. We know that $h(xu)\setminus \{x\}$ and $h(yv)\setminus \{y\}$ are disjoint. We need to show that $x\not\in h(yv)$ and $y\not\in h(xu)$. Suppose that $x\in h(yv)$, then $x\not\in V(K)$ (since $v$ is an induced-available neighbor of $y$ with respect to $K$). Thus, $x\in A'_z$ for some $z\in C$ and $v\in A'_y$. Hence, $h(zx)\setminus \{z\}$ and $h(yv)\setminus \{y\}$ are disjoint, but $x$ is in both sets, which is a contradiction. This shows that $h(xu) \cap h(yv)=\emptyset$. 
	
	Now, suppose that for some $h\in E(\mathcal{H})$, we have $|h\cap \cl(\widehat{K})|\geq 2$ and $h\neq h(e)$, for any $e\in E(\widehat{K})$. So, $h$ intersects some hyperedges $h(xu)$ and $h(yv) $ for some $xu,yv\in E(\widehat{K})$. If $xu,yv\in E(K)$, then this is impossible because $K$ is induced-good. If $xu\not\in E(K)$ and $yv\in E(K)$, then this is impossible because $u$ is an induced-available neighbor of $x$ with respect to $K$. Finally, suppose that $xu, yv\not\in E(K)$. Then, $u\in A'_x$ and $v\in A'_y$, so either $h$ intersects $h(xu)$ in just $x$, or $h$ intersects $h(yv) $ in just $y$. By symmetry, assume that the former occurs. If $x\in V(K)$, then this is in contradiction with  the fact that $v$ is an induced-available neighbor of $y$ and if $x\not\in V(K)$, then $x\in A'_z$, for some $z\in C$ and so $h$ intersects both $h(zx)\setminus\{z\}$ (in $x$) and  $h(yv)\setminus \{y\} $. This is in contradiction with the choice of the sets $A'_z$, $z\in C$.
Hence, $\widehat{K}$ is induced-good and the proof is complete. 
	
	Now, we check that Inequality~\eqref{eq:aharoni3} holds. Fix a subset $I'\subseteq C'$ and let $T$ be a transversal set for $N_{\tilde{\mathcal{H}}}(I')$ where $|T|=\tau(N_{\tilde{\mathcal{H}}} (I'))$. For the contrary, suppose that $|T|\leq (2s-3)(D'|I'|-1)$.  Also, let $I=\{x\in C: x'\in I'\}$ and $A:= I\cup T$. Also, let $\mathcal{E}= \{h(xy)\in E(\mathcal{H}): xy\in E(K^*), x\in I \}$. For every hyperedge $h(xy)\in \mathcal{E}$, $h(xy)\setminus \{x\}$ intersects $T$. So, $h(xy)$ intersects $A$ in at least two vertices including $x$ itself.
	
	On the other hand,
	\begin{align*}
	|A|&= |I|+|T|\leq |I|+(2s-3) (D'|I|-1) \\
	&<  2sD'|I|\leq 2sD'|C|\\
	& \leq \alpha N.
	\end{align*}
	
	Therefore, by (P4) we have,
	\begin{align*}
	\dfrac{10^5Ds^3\ln s |I|}{2} \leq |\mathcal{E}|\leq 2|A|< 4sD' |I|.
	\end{align*}
	This is a contradiction. So, Inequality \eqref{eq:aharoni3} holds.
	
	Finally, we check the properties of the graph $L$. First, suppose that $e_L(A_x,A_y)\geq 2$, so there are two hyperedges $h,h'$ such that $h$ intersects $h(xu)\setminus \{x\}$ and $h(yv)\setminus \{y\}$ and  $h'$ intersects $h(xu')\setminus \{x\}$ and $h(yv')\setminus \{y\}$, for some $u,u'\in A_x$ and $v,v'\in A_y$. But this comes to a Berge cycle of length at most $6$, as $(h(xu), h, h(yv), h(yv'), h', h(xu'), h(xu))$, a contradiction with (P3) in Lemma~\ref{lem:hypergraph}. Hence, $e_L(A_x,A_y)\leq 1$. 
	
	Now, we prove that $L$ is $3s^2$-degenerate.  Suppose that $U$ is a subset of $V(L)$. Let $A= \bigcup_{x\in C, u\in A_x\cap U} h(xu)\setminus \{x\}$. Then, for every edge $uv\in E_L(U)$, where $u\in A_x\cap U$ and $v\in  A_y\cap U$, there exists a hyperedge $h\in E(\mathcal{H})$ which intersects both $h(xu)\setminus \{x\}$ and $h(yv)\setminus \{y\}$. Therefore, $h$ intersects $A$ in at least two vertices. On the other hand, we have $|A|\leq s|U|\leq s|V(L)|\leq s|C|D' \leq \alpha N$. So, by (P4), we have 
	\[e_L(U)\leq \sum_{h\in E(\mathcal{H})} \binom{|h\cap A|}{2}\leq \dfrac{s}{2}\sum_{\stackrel{h\in E(\mathcal{H})}{ |h\cap A|\geq 2}} |h\cap A|\leq \dfrac{5}{4} s|A|\leq \dfrac{5}{4} s^2|U|. \]
	This means that the average degree of $L[U]$ is at most $3s^2$ and so $L$ is $3s^2$-degenerate. 
	
	Finally, we compute the maximum degree of $L$, $\Delta_L$. Let $u\in A_x$ be a vertex in $V(L) $. Note that there are at most $s\Delta_{\mathcal{H}}$ hyperedges intersecting $h(xu)$. Each of these hyperedges contains $s$ vertices. So, $\Delta_L\leq s^2\Delta_{\mathcal{H}}$ (note that the upper bound is not $s^2\Delta_{\mathcal{H}}^2$ because the sets $h^{x'}_i\setminus \{x'\}$ and $h^{y'}_j\setminus \{y'\}$ are disjoint). Hence,
	\[
	\Delta_L\leq s^2\Delta_{\mathcal{H}} \leq 8cs^3 \leq (10s)^8D.
	\]  
This completes the proof.	
\end{proof}

\section{Proof of Theorem~\ref{thm:induced_subd}} \label{sec:mainproof}
Recall that  $\ell$ is equal to $6$ in the even case and equal to $5$ in the general case. Also, 
 $L_1=\lfloor (\ell-1)/2\rfloor $ and $L_2= \lceil (\ell+1)/2 \rceil  $.
In this section, we want to show how we can find an induced-good embedding of $H^{\sigma'}$ in $G'$ (defined in Lemma~\ref{lem:subgraph}) such that $\sigma(e)/L_2\leq \sigma'(e)\leq \sigma(e)/L_1$ for all $e\in E(H)$. This, along with Lemma~\ref{lem:sigma'tosigma}, implies the existence of an induced monochromatic copy of $H^\sigma$ in $\mathcal{H}(F)$ and Theorem~\ref{thm:induced_subd} is proved.

Note that, by Lemma~\ref{lem:subgraph}, $G'$ has the average degree at least $(10^7/2)\, Ds^4\ln s $, so there is an induced subgraph $G''$ of $G'$ with $\delta(G'')\geq (10^7/4)\, D s^4\ln s$. In addition, the sparsity property (Q2) is also valid for $G''$.

We fix an ordering on $E(H)$ as $e_1,\ldots, e_m$ and we embed
$H^{\sigma'}$ in $G'$ in $m$ steps. Let us call the vertices of
$V(H)$, the \textit{original vertices} of $H^{\sigma'}$ and the
vertices that appear after subdivisions the \textit{internal vertices}
of $H^{\sigma'}$. All original vertices are embedded in the graph
$G''$. In Step $i$, if $e_i=ab\in E(H)$, we embed a path of
length~$\sigma'(e)$ between~$a$ and~$b$ in~$G'$ following an approach
in~\cite{Sudakov-cycle}. This step consists of three phases, as
follows (see Figure~\ref{fig:scheme}).

\begin{figure}

\tikzset{every picture/.style={line width=0.75pt}} %
\begin{center}
\begin{tikzpicture}[x=0.75pt,y=0.75pt,yscale=-1,xscale=1]
	
	\draw  [fill={rgb, 255:red, 0; green, 0; blue, 0 }  ,fill opacity=1 ] (283,27.5) .. controls (283,23.91) and (285.91,21) .. (289.5,21) .. controls (293.09,21) and (296,23.91) .. (296,27.5) .. controls (296,31.09) and (293.09,34) .. (289.5,34) .. controls (285.91,34) and (283,31.09) .. (283,27.5) -- cycle ;
	\draw  [fill={rgb, 255:red, 0; green, 0; blue, 0 }  ,fill opacity=1 ] (337,83.5) .. controls (337,79.91) and (339.91,77) .. (343.5,77) .. controls (347.09,77) and (350,79.91) .. (350,83.5) .. controls (350,87.09) and (347.09,90) .. (343.5,90) .. controls (339.91,90) and (337,87.09) .. (337,83.5) -- cycle ;
	\draw  [fill={rgb, 255:red, 0; green, 0; blue, 0 }  ,fill opacity=1 ] (229,84.5) .. controls (229,80.91) and (231.91,78) .. (235.5,78) .. controls (239.09,78) and (242,80.91) .. (242,84.5) .. controls (242,88.09) and (239.09,91) .. (235.5,91) .. controls (231.91,91) and (229,88.09) .. (229,84.5) -- cycle ;
	\draw    (289.5,27.5) -- (189,132) ;
	\draw    (289.5,27.5) -- (390,131) ;
	\draw  [fill={rgb, 255:red, 0; green, 0; blue, 0 }  ,fill opacity=1 ] (383.5,131) .. controls (383.5,127.41) and (386.41,124.5) .. (390,124.5) .. controls (393.59,124.5) and (396.5,127.41) .. (396.5,131) .. controls (396.5,134.59) and (393.59,137.5) .. (390,137.5) .. controls (386.41,137.5) and (383.5,134.59) .. (383.5,131) -- cycle ;
	\draw  [fill={rgb, 255:red, 0; green, 0; blue, 0 }  ,fill opacity=1 ] (182.5,132) .. controls (182.5,128.41) and (185.41,125.5) .. (189,125.5) .. controls (192.59,125.5) and (195.5,128.41) .. (195.5,132) .. controls (195.5,135.59) and (192.59,138.5) .. (189,138.5) .. controls (185.41,138.5) and (182.5,135.59) .. (182.5,132) -- cycle ;
	\draw    (235.5,84.5) -- (343.5,83.5) ;
	\draw  [fill={rgb, 255:red, 0; green, 0; blue, 0 }  ,fill opacity=1 ] (200.5,116.25) .. controls (200.5,114.18) and (202.18,112.5) .. (204.25,112.5) .. controls (206.32,112.5) and (208,114.18) .. (208,116.25) .. controls (208,118.32) and (206.32,120) .. (204.25,120) .. controls (202.18,120) and (200.5,118.32) .. (200.5,116.25) -- cycle ;
	\draw  [fill={rgb, 255:red, 0; green, 0; blue, 0 }  ,fill opacity=1 ] (215.5,101.25) .. controls (215.5,99.18) and (217.18,97.5) .. (219.25,97.5) .. controls (221.32,97.5) and (223,99.18) .. (223,101.25) .. controls (223,103.32) and (221.32,105) .. (219.25,105) .. controls (217.18,105) and (215.5,103.32) .. (215.5,101.25) -- cycle ;
	\draw  [fill={rgb, 255:red, 0; green, 0; blue, 0 }  ,fill opacity=1 ] (259.5,84.25) .. controls (259.5,82.18) and (261.18,80.5) .. (263.25,80.5) .. controls (265.32,80.5) and (267,82.18) .. (267,84.25) .. controls (267,86.32) and (265.32,88) .. (263.25,88) .. controls (261.18,88) and (259.5,86.32) .. (259.5,84.25) -- cycle ;
	\draw  [fill={rgb, 255:red, 0; green, 0; blue, 0 }  ,fill opacity=1 ] (286.5,84.25) .. controls (286.5,82.18) and (288.18,80.5) .. (290.25,80.5) .. controls (292.32,80.5) and (294,82.18) .. (294,84.25) .. controls (294,86.32) and (292.32,88) .. (290.25,88) .. controls (288.18,88) and (286.5,86.32) .. (286.5,84.25) -- cycle ;
	\draw  [fill={rgb, 255:red, 0; green, 0; blue, 0 }  ,fill opacity=1 ] (312.5,84.25) .. controls (312.5,82.18) and (314.18,80.5) .. (316.25,80.5) .. controls (318.32,80.5) and (320,82.18) .. (320,84.25) .. controls (320,86.32) and (318.32,88) .. (316.25,88) .. controls (314.18,88) and (312.5,86.32) .. (312.5,84.25) -- cycle ;
	\draw  [fill={rgb, 255:red, 0; green, 0; blue, 0 }  ,fill opacity=1 ] (361.5,106.25) .. controls (361.5,104.18) and (363.18,102.5) .. (365.25,102.5) .. controls (367.32,102.5) and (369,104.18) .. (369,106.25) .. controls (369,108.32) and (367.32,110) .. (365.25,110) .. controls (363.18,110) and (361.5,108.32) .. (361.5,106.25) -- cycle ;
	\draw  [fill={rgb, 255:red, 0; green, 0; blue, 0 }  ,fill opacity=1 ] (318.5,61.25) .. controls (318.5,59.18) and (320.18,57.5) .. (322.25,57.5) .. controls (324.32,57.5) and (326,59.18) .. (326,61.25) .. controls (326,63.32) and (324.32,65) .. (322.25,65) .. controls (320.18,65) and (318.5,63.32) .. (318.5,61.25) -- cycle ;
	\draw  [fill={rgb, 255:red, 0; green, 0; blue, 0 }  ,fill opacity=1 ] (304.5,47.25) .. controls (304.5,45.18) and (306.18,43.5) .. (308.25,43.5) .. controls (310.32,43.5) and (312,45.18) .. (312,47.25) .. controls (312,49.32) and (310.32,51) .. (308.25,51) .. controls (306.18,51) and (304.5,49.32) .. (304.5,47.25) -- cycle ;
	\draw  [fill={rgb, 255:red, 0; green, 0; blue, 0 }  ,fill opacity=1 ] (271.5,42.25) .. controls (271.5,40.18) and (273.18,38.5) .. (275.25,38.5) .. controls (277.32,38.5) and (279,40.18) .. (279,42.25) .. controls (279,44.32) and (277.32,46) .. (275.25,46) .. controls (273.18,46) and (271.5,44.32) .. (271.5,42.25) -- cycle ;
	\draw  [fill={rgb, 255:red, 0; green, 0; blue, 0 }  ,fill opacity=1 ] (252.5,61.25) .. controls (252.5,59.18) and (254.18,57.5) .. (256.25,57.5) .. controls (258.32,57.5) and (260,59.18) .. (260,61.25) .. controls (260,63.32) and (258.32,65) .. (256.25,65) .. controls (254.18,65) and (252.5,63.32) .. (252.5,61.25) -- cycle ;
	\draw    (189,132) -- (189,205) ;
	\draw   (189,205) -- (224,297) -- (154,297) -- cycle ;
	\draw    (210,261) -- (167,261) ;
	\draw    (390,130) -- (390,203) ;
	\draw   (390,203) -- (425,295) -- (355,295) -- cycle ;
	\draw    (411,259) -- (368,259) ;
	\draw    (192,279) .. controls (222,330) and (354,333) .. (387,280) ;
	\draw    (147.04,137) -- (147.96,205) ;
	\draw [shift={(148,208)}, rotate = 269.23] [fill={rgb, 255:red, 0; green, 0; blue, 0 }  ][line width=0.08]  [draw opacity=0] (8.93,-4.29) -- (0,0) -- (8.93,4.29) -- cycle    ;
	\draw [shift={(147,134)}, rotate = 89.23] [fill={rgb, 255:red, 0; green, 0; blue, 0 }  ][line width=0.08]  [draw opacity=0] (8.93,-4.29) -- (0,0) -- (8.93,4.29) -- cycle    ;
	\draw    (148,210) -- (148,257) ;
	\draw [shift={(148,260)}, rotate = 270] [fill={rgb, 255:red, 0; green, 0; blue, 0 }  ][line width=0.08]  [draw opacity=0] (8.93,-4.29) -- (0,0) -- (8.93,4.29) -- cycle    ;
	\draw [shift={(148,207)}, rotate = 90] [fill={rgb, 255:red, 0; green, 0; blue, 0 }  ][line width=0.08]  [draw opacity=0] (8.93,-4.29) -- (0,0) -- (8.93,4.29) -- cycle    ;
	\draw    (148,263) -- (148,294) ;
	\draw [shift={(148,297)}, rotate = 270] [fill={rgb, 255:red, 0; green, 0; blue, 0 }  ][line width=0.08]  [draw opacity=0] (8.93,-4.29) -- (0,0) -- (8.93,4.29) -- cycle    ;
	\draw [shift={(148,260)}, rotate = 90] [fill={rgb, 255:red, 0; green, 0; blue, 0 }  ][line width=0.08]  [draw opacity=0] (8.93,-4.29) -- (0,0) -- (8.93,4.29) -- cycle    ;
	\draw    (431.04,132) -- (431.96,200) ;
	\draw [shift={(432,203)}, rotate = 269.23] [fill={rgb, 255:red, 0; green, 0; blue, 0 }  ][line width=0.08]  [draw opacity=0] (8.93,-4.29) -- (0,0) -- (8.93,4.29) -- cycle    ;
	\draw [shift={(431,129)}, rotate = 89.23] [fill={rgb, 255:red, 0; green, 0; blue, 0 }  ][line width=0.08]  [draw opacity=0] (8.93,-4.29) -- (0,0) -- (8.93,4.29) -- cycle    ;
	\draw    (432,205) -- (432,252) ;
	\draw [shift={(432,255)}, rotate = 270] [fill={rgb, 255:red, 0; green, 0; blue, 0 }  ][line width=0.08]  [draw opacity=0] (8.93,-4.29) -- (0,0) -- (8.93,4.29) -- cycle    ;
	\draw [shift={(432,202)}, rotate = 90] [fill={rgb, 255:red, 0; green, 0; blue, 0 }  ][line width=0.08]  [draw opacity=0] (8.93,-4.29) -- (0,0) -- (8.93,4.29) -- cycle    ;
	\draw    (432,258) -- (432,289) ;
	\draw [shift={(432,292)}, rotate = 270] [fill={rgb, 255:red, 0; green, 0; blue, 0 }  ][line width=0.08]  [draw opacity=0] (8.93,-4.29) -- (0,0) -- (8.93,4.29) -- cycle    ;
	\draw [shift={(432,255)}, rotate = 90] [fill={rgb, 255:red, 0; green, 0; blue, 0 }  ][line width=0.08]  [draw opacity=0] (8.93,-4.29) -- (0,0) -- (8.93,4.29) -- cycle    ;
	
	\draw (168,106) node [anchor=north west][inner sep=0.75pt]    {$a$};
	\draw (395,101) node [anchor=north west][inner sep=0.75pt]    {$b$};
	\draw (260,289) node [anchor=north west][inner sep=0.75pt]    {$O_{k} ,_{D}( 1)$};
	\draw (74,267) node [anchor=north west][inner sep=0.75pt]    {$O_{k} ,_{D}( 1)$};
	\draw (443,262) node [anchor=north west][inner sep=0.75pt]    {$O_{k} ,_{D}( 1)$};
	\draw (81,221) node [anchor=north west][inner sep=0.75pt]    {$\log_{D}( n)$};
	\draw (98,148) node [anchor=north west][inner sep=0.75pt]    {$\dfrac{\sigma ( e)}{2L_{2}}$};
	\draw (444,216) node [anchor=north west][inner sep=0.75pt]    {$\log_{D}( n)$};
	\draw (444,146) node [anchor=north west][inner sep=0.75pt]    {$\dfrac{\sigma ( e)}{2L_{2}}$};
	\draw (162,271) node [anchor=north west][inner sep=0.75pt]    {$R_{1}$};
	\draw (395,269) node [anchor=north west][inner sep=0.75pt]    {$R_{2}$};
	\draw (206,222) node [anchor=north west][inner sep=0.75pt]    {$T'_{1}$};
	\draw (350,221) node [anchor=north west][inner sep=0.75pt]    {$T'_{2}$};
	\draw (196,162) node [anchor=north west][inner sep=0.75pt]    {$P_{1}$};
	\draw (362,156) node [anchor=north west][inner sep=0.75pt]    {$P_{2}$};
	\draw (222,310) node [anchor=north west][inner sep=0.75pt]    {$P'$};
	
\end{tikzpicture}
\end{center}

	\caption{Scheme of the embedded graph. }\label{fig:scheme} 
\end{figure}
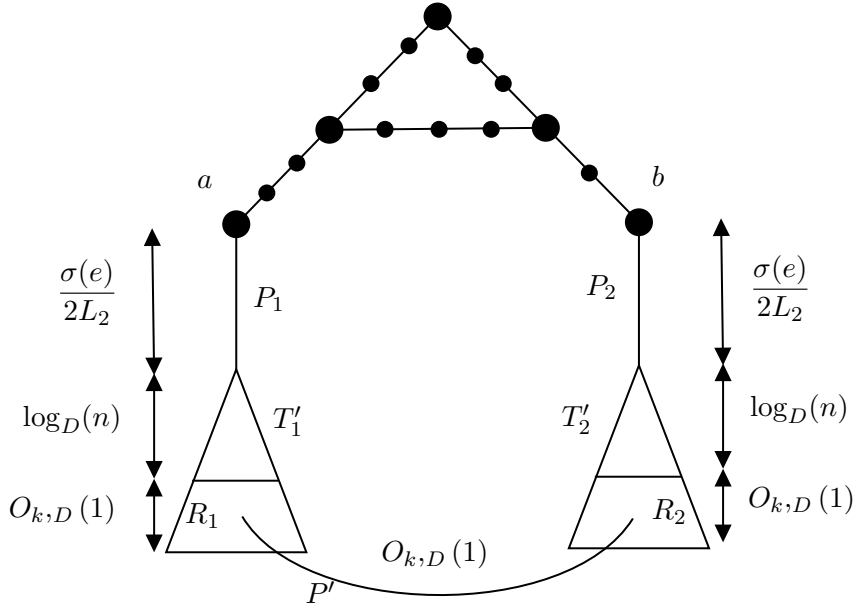

Set $d'=10^5Ds^3\ln s $ and $d=20\times 10^5D s^4\ln s $. 
\begin{enumerate}
	\item In the first phase, we grow two disjoint trees $T_1$ and $T_2$ from the roots $a$ and $b$ in $G''$, where $T_j$, $j\in\{1,2\}$, consists of a path $P_j$ of length $\lfloor \sigma(e)/(2L_2)\rfloor $ along with a tree $T_j'$ with maximum degree $D$ and  $|V(T_j')|=\lceil \alpha N/(50s)\rceil$.  
	
	Now, let $R_j$ be the set of vertices of $T'_j$ in the last $\lceil \log_D (\alpha N/(50s))\rceil-\lfloor \log_D n\rfloor $ levels. Therefore, the height of $T'_j- R_j$ is at most $\log_{D}(n)$ and so, $v(T'_j-R_j)\leq 2n $. Also, $|R_j|\geq \alpha N/(60s)$.
	\item In the second phase, we find a path $P'$ of length at most ${10^9 sk\ln^{2}(sD)}$ in $G'$ from a vertex of $R_1$ to a vertex of $R_2$ and complete the path between $a$ and $b$, and
	\item In the third phase, we delete extra vertices of $T'_1$ and $T'_2$ (i.e. vertices of $T'_1$ and $T'_2$ except the $ab$-path) from the embedded subgraph. 
\end{enumerate} 

\subsection{Phase 1. Growing the trees}

We always keep two subgraphs $J_i$ and $K_i$ of $G'$ and a set of vertices $F_i\subseteq V(G'-K_i)$, where $J_i$ is a subgraph of $K_i$ and the following properties hold: (see Figure~\ref{fig:graphK})

\begin{itemize}
	\item[(A1)] $J_i$ is the already embedded subgraph and $K_i$ is a supergraph of $J_i$. Both graphs are induced-good subgraphs of $G'$ and consist of the graph $H_{i}:=H^{\sigma'}[e_1,\ldots, e_{i-1}]$ and some disjoint trees that are rooted at the original vertices of $H_{i}$. All vertices of these trees are in $V(G'')$. 
	\item[(A2)] 
	Any non-internal non-leaf vertex $v$ of $K_i$ has at least $D-\deg_{H_{i}}(v)$ children in $K_i\setminus J_i$ (if $v\not\in V(H_i)$, then $\deg_{H_i}(v)=0$). 
	\item[(A3)] All vertices in $F_i$ are forbidden and will never be used in the embedded subgraph. Also, $\cl(K_i)\setminus V(K_i)\subseteq F_i$. 
	\item[(A4)] We have $C_i:=C^{\mathrm{ind}}_{d,d'}(J_i)=C^{\mathrm{ind}}_{d,d'}(K_i)$. Also, if $N_i$ is the set of non-leaf vertices of $K_i$, then $N_i\subseteq C_i\subseteq F_i\cup N_i$. 
\end{itemize}

Now, suppose that we have already embedded the graph $H_{i}:=H^{\sigma'}[e_1,\ldots,e_{i-1}]$ and kept the graphs $J_i,K_i$ and the subset $F_i$ with the above properties. First, assume that both $a,b$ are in $V(H_i)$. We are going to grow trees $T_1$ and $T_2$ from $a$ and $b$ by adding leaves one by one. Let $v$ be the vertex to which we are going to add a leaf (in the beginning $v=a$ or $v=b$). We consider two cases: 

(1) If $v\in C_i$, then by (A4), $v$ is a non-leaf of $K_i$ and by (A2), $v$ has at least one neighbor $u$ in $K_i\setminus J_i$. So, we can extend $J_i$ by adding the edge $uv$. Thus, the critical vertices do not change\footnote{That holds because the edges of $K_i$ are among $d'$ edges chosen in the construction of the critical set (Definition~\ref{def:induced-critical}).}. Now, we update $J_{i}=J_i+uv$, $K_{i}:=K_i$ and $F_{i}:=F_i$.

\begin{figure}
	\hspace*{3cm}\includegraphics[scale=.19]{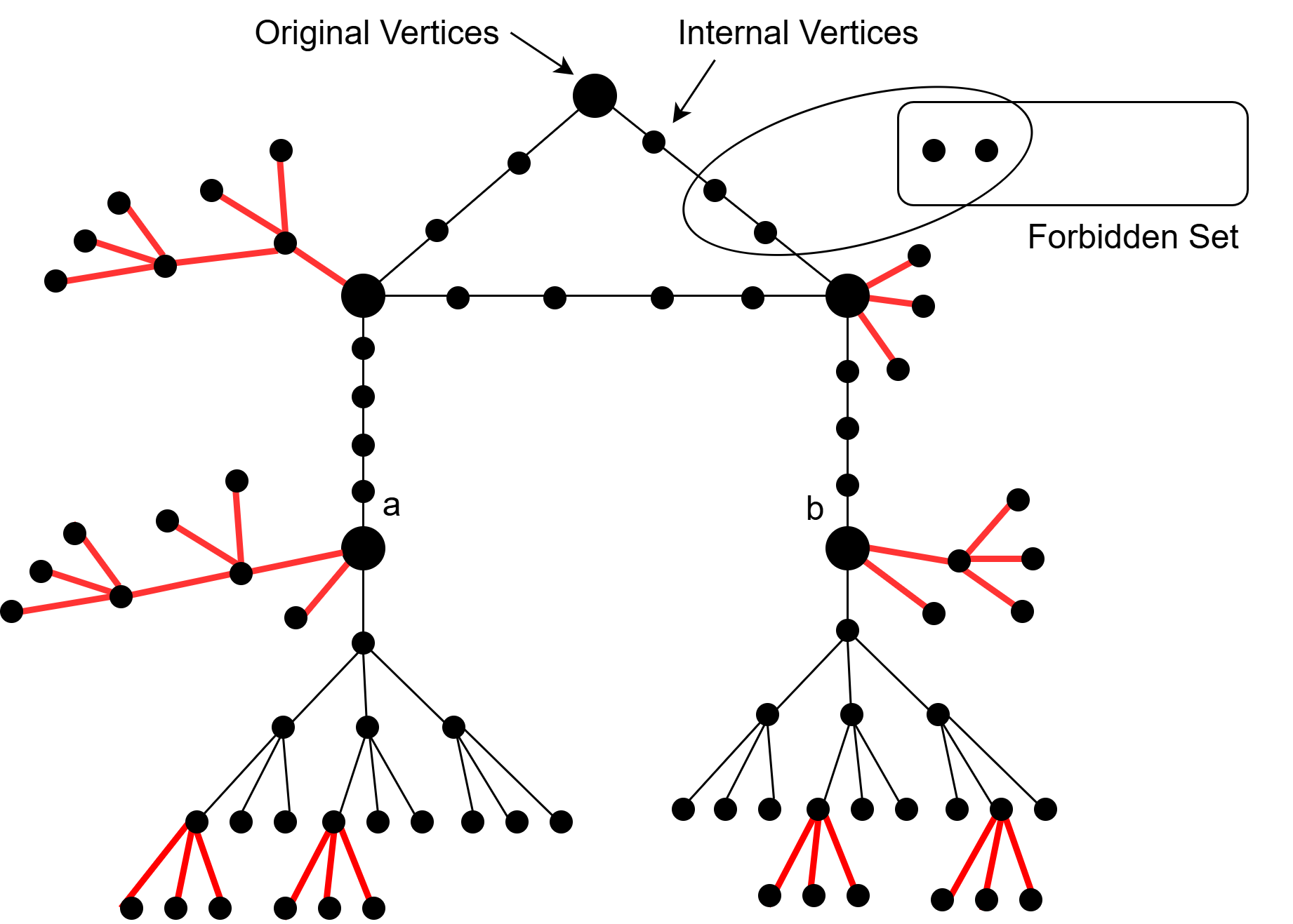}
	\caption{Subgraphs $J_i$ and $K_i$ and the forbidden set $F_i$. Red (bold) edges are edges of $K_i\setminus J_i$.}\label{fig:graphK} 
\end{figure}

(2) Now, suppose that $v\not\in C_i$. So, by (A4), it has at most $d-1$ neighbors which are not induced-available with respect to $K_i$ and since $\delta(G'')\geq d+d'$, it has at least $d'+1=10^5Ds^3\ln s+1$ induced-available neighbors with respect to $K_i$. If we set $J'$ to be the graph obtained from $J_i$ by adding $d'$ of  induced-available neighbors of $v$, then $C':=C^{\mathrm{ind}}_{d,d'}(J')$ has some new vertices including $v$. Let $\widehat{C}=C'\setminus C_i$. 
Note that
\begin{align*}
e(J')&\leq e(H^\sigma)+v(T_1')+v(T_2')+d' \\
& \leq nD+\frac{\alpha N}{25s}+2+d' \\
&\leq \frac{\alpha N}{22s}-d'.
\end{align*}
So, by Lemma~\ref{lem:induced-critical}, we have $|\widehat{C}|\leq \alpha N/(22sd')$. Moreover, for each vertex $u\in \widehat{C}$ (since $u\not\in C_i$) there exist at least $d'$ induced-available neighbors with respect to $K_i$. Thus, applying Lemma~\ref{lem:extend-induced} with $K:=K_i$ and $C:=\widehat{C}$, we can obtain a graph $\widehat{K}$ such that all leaves of $\widehat{K}$ are in $V(G'')\setminus C'$ and all non-leaves of $\widehat{K}$ have $D$ children. Now, let $u$ be a neighbor of $v$ in $\widehat{K}\setminus J_i$ and update $J_{i}:=J_i+ vu$, $K_i:=\widehat{K}$ and $F_{i}:= F_i\cup (\cl(\widehat{K})\cup \widehat{C})\setminus (V(\widehat{K})) $. It is clear that all conditions (A1)-(A4) remain valid. 

Now, suppose that either $a$ or $b$, say $a$, is not in $V(H_{i})$, then we choose an arbitrary vertex in $V(G'')\setminus (V(K_i)\cup F_i)$ as the vertex $a$ (in Phase 3, we will see that this set is non-empty). By (A4), $a$ is not in $C_i$. By a similar argument as above, using Lemma~\ref{lem:extend-induced} with $K=K_i\cup \{a\}$ and $C=\widehat{C}$, we can obtain $\widehat{K}$ and then update $J_{i}:=J_i+ au$, $K_i:=\widehat{K}$ and $F_{i}:= F_i\cup (\cl(\widehat{K})\cup \widehat{C})\setminus (V(\widehat{K})) $, where $u$ is a neighbor of $a$ in $\widehat{K}$. Again, one may check that (A1)-(A4) are still valid.

\subsection{Phase 2. Connecting the trees}
Now, we want to find an induced-good path in $G'$ from a vertex of $R_1$ to a vertex of $R_2$ and add it to $J_i$. First, for every vertex $v$ in $V(K_i)\setminus C^{\mathrm{ind}}_{d,d'}(J_i\setminus(R_1\cup R_2))$, remove all children of $v$ from $K_i$. Now, since $v(J_i\setminus (R_1\cup R_2))\leq 5n$, by Lemma~\ref{lem:induced-critical} and (A4), we have $e(K_i)\leq 6nD$. Also, note that 

$$|\cl_3(K_i)|\leq  |\cl_2(K_i)|\,\Delta_G s\leq |\cl(K_i)|\,\Delta_G^2 s^2\leq 6nD\Delta_G^2 s^3\leq  384Dc^2 s^5\, n.$$

Now for $j\in \{1, 2\}$, set $A_j:=R_j$. As long as, $|A_j|\leq |V(G')|/2$, let $A_j:= A_j\cup N_{G'}(A_j)\setminus \cl_3(K_i)$. Since $G'$ is $\gamma$-expander with $\gamma\geq {(10^4sk\log_2 (sD))^{-1}}$, after $t$ steps, we have 

\begin{align*}
|A_j|&\geq (1+\gamma)^t |R_j|-\dfrac{(1+\gamma)^t-1}{\gamma}|\cl_3(K_i)|\geq \dfrac{(1+\gamma)^t}{\gamma} (\gamma|R_j|-|\cl_3(K_i)|)\\
& \geq \dfrac{(1+\gamma)^t}{\gamma} \left(\dfrac{\gamma\alpha N}{60s}-  384Dc^2 s^5\, n\right)\geq  \dfrac{\alpha N}{100s}(1+\gamma)^t.
\end{align*}

Note that in the last inequality, we assume $\gamma\alpha N\geq 6\times 10^4 Dc^2 s^6n$.

Therefore, after at most $t=\lceil \ln(50s \alpha^{-1})/ \ln(1+\gamma) \rceil \leq \lceil 2\ln(50s \alpha^{-1})/\gamma\rceil $ steps, $|A_j|$ passes $|V(G')|/2$. Hence, $A_1$ and $A_2$ contain a vertex $x$ in common. This implies that there is a path $P'$ of length at most $2\lceil 2\ln(50s \alpha^{-1})/\gamma\rceil \leq 10^9  sk \ln^2 (sD)$ between a vertex $x_1$ in $R_1$ and a vertex $x_2$ in $R_2$. 

Now, set $J_i:=J_i\cup P'$. First,  we show that $ J_i $ is an induced-good subgraph. Let $e_1$ and $e_2$ be two nonadjacent edges in $J_i$. If either $e_1$ or $e_2$ is in $E(J_i\setminus (R_1\cup R_2\cup P')) $, then $h(e_1)\cap h(e_2)=\emptyset$, since $J_i$ before adding $P'$ was induced-good and in the construction of $P'$, $V(P')$ does not intersect $\cl_2(K_i)$. Now, suppose that both $e_1$ and $e_2$ belong to $E(R_1\cup R_2\cup P') $. As the height of $R_j$ is $\lceil \log_D (\alpha N/(50s))\rceil-\lfloor \log_D n\rfloor\leq 10^4 \log_{D} (sD) $ and $v(P')\leq  10^9 sk \ln^2 (sD)$, if $h(e_1)\cap h(e_2)\neq \emptyset$, then there would be a Berge cycle in $\mathcal{H}$ of length at most $10^{10}sk\ln^2(sD)$ which is in contradiction with Property (P3).  

Now, we have to check the second condition of induced-goodness. Suppose that $h\in E(\mathcal{H})$ and $|h\cap \cl(J_i)|\geq 2$, say $h\cap h(e_1)\neq \emptyset$ and $h\cap h(e_2)\neq \emptyset$ for some $e_1,e_2\in E(J_i)$.  We have to show that $h=h(e)$ for some $e\in E(J_i)$. Again if either $e_1$ or $e_2$ is in $E(J_i\setminus (R_1\cup R_2\cup P')) $, then $h=h(e)$ for some $e\in E(J_i)$ since $J_i$ before adding $P'$ was induced-good and in the construction of $P'$, $V(P')$ does not intersect $\cl_3(K_i)$. Now, if both $e_1$ and $e_2$ belong to $E(R_1\cup R_2\cup P') $, then, by a similar argument as above, $\mathcal{H}$ contains a Berge cycle of length at most $10^{10}sk\ln^2(sD)$ which is in contradiction with Property (P3).  

Finally, we are going to update $K_i$ and $F_i$ as follows. Let $C^*:=C^{\mathrm{ind}}_{d, d'}(J_i)\setminus \big(V(P')\cup C^{\mathrm{ind}}_{d, d'}(J_i\setminus P')\big)  $. For every vertex $v\in C^*$ there exist $d'+1=10^5Ds^3\ln s+1$ induced-available neighbors with respect to $K_i$. At most one of these neighbors is not induced-available for $K_i\cup P'$, as otherwise there would be a Berge cycle of length at most $v(P')+6\leq 10^9 sk\ln^2(sD)+6$, which contradicts Property (P3). So, for each vertex $v\in C^* $ choose $d'$ induced-available neighbors with respect to $K_i\cup P'$ and apply Lemma \ref{lem:extend-induced} by setting $K=K_i\cup P'$ and $C=C^*$ to extend $K_i$ to $\widehat{K}$. (Note that $e(J_i)\leq nD+2\lceil\alpha N/(50s)\rceil+v(P')\leq \alpha N/(22s)-d'$ and thus by Lemma~\ref{lem:induced-critical}, we have $|C^*|\leq \alpha N/(22sd') $.) Now, set $K_i:=K_i\cup P'\cup \widehat{K}$ and $F_i:=F_i\cup (\cl(K_i)\cup C^{\mathrm{ind}}_{d,d'}(J_i))\setminus V(K_i)$.

Hence, we have a path $P$ starting from $a$, going through $T_1$, continuing from a vertex in $R_1$ through the path $P'$ and then traversing to $b$ through $T_2$. Note that

\begin{align*}
\dfrac{\sigma(e)}{L_2} \leq v(P)&\leq \dfrac{\sigma(e)}{L_2}+ 2\left\lceil \log_{D}\left(\dfrac{\alpha N}{50s} \right)\right\rceil+v(P')\\
&\leq \dfrac{\sigma(e)}{L_2}+2\log_{D}(n)+10^4\log_D(sD)+ 10^9 sk\ln^2(sD)\\
&\leq \dfrac{\sigma(e)}{L_2}+2\log_{D}(n)+1.1\times 10^{9} sk\ln^2(sD) \\
&\leq \dfrac{\sigma(e)}{L_1},
\end{align*}
where the last inequality holds because $\sigma(e)\geq \eta\log_D(n)+10^{10} sk\ln^2(sD)$, where $\eta=8$ for the even case and $\eta=12$ in the general case. 

\subsection{Phase 3. Rolling backward}
In the final phase, we are going to clean the graph $J_i$ and remove extra vertices in $K_i\cup F_i$. First, set $H_{i+1}:=H_i+P$, $U:= V(J_i)\setminus V(H_{i+1})$, and $J_{i+1}:=J_i- U=H_{i+1}$. Let $C:= C^{\mathrm{ind}}_{d,d'}(J_i)\setminus C^{\mathrm{ind}}_{d,d'}(J_{i+1})$ and for every $v\in C$, remove all children of $v$ from $K_i$. Now, let $K_{i+1}$ be the graph obtained from $K_i$ by removing all connected components which do not intersect $J_{i+1}$. Finally, define $F_{i+1}:= (\cl(K_{i+1})\cup C^{\mathrm{ind}}_{d,d'}(J_{i+1})) \setminus V(K_{i+1})$.
Note that after this cleaning phase, we have $|V(J_{i+1})|=|V(H_{i+1})|\leq n$. Also, by Lemma~\ref{lem:induced-critical} and (A4), $|E(K_{i+1})|\leq D(1+1/d')n$ and $|F_{i+1}|\leq [sD(1+2/d')]n$. Since $|V(G'')|\geq \alpha N$, we have $V(G'')\setminus (V(K_{i+1})\cup F_{i+1})$ is non-empty.

\section{Concluding remark}

In Theorems~\ref{thm:induced_subd} and~\ref{thm:noninduced_subd} we
distinguish two cases: the case in which the parity of~$\sigma(e)$
($e\in E(H)$) is arbitrary and the case in which~$\sigma(e)$ is even
for every~$e\in E(H)$.  Naturally, in the latter case~$H^\sigma$ is a
bipartite graph.  In a companion paper~\cite{JKM}, we consider the
case in which~$H^\sigma$ is bipartite, \textit{without} imposing the
hypothesis that~$\sigma(e)$ is even for every~$e$.  In~\cite{JKM}, we
obtain the following result.

\begin{theorem}
  \label{thm:main}
  Let $k, \, D\geq 2$ be integers, and let $H$ be a graph with maximum
  degree~$D$.  Suppose~$H^\sigma$ is a subdivision of $H$
  with~$\sigma(e)\geq 2\log_{D-1} n$ for every $e \in E(H)$,  where
  $n=v(H^\sigma)$.  Furthermore, suppose~$H^\sigma$ is bipartite.
  Then
  \[
    \widehat{R}(H^\sigma, k) \leq k^{\,400 D\log D} \, n.
  \]
\end{theorem}

Our approach in~\cite{JKM} is entirely different from the approach in
this paper.

\begin{remark}
  The results in this paper are in fact universality statements.  Let
  us say that a graph~$G$ is $k$-\textit{partition induced universal}
  for a class of graphs~$\mathcal{H}$ if, under every $k$-edge
  coloring of~$G$, there is some fixed color such that for every
  member~$H$ of~$\mathcal H$, the graph~$G$ contains an induced copy
  of~$H$ that is monochromatic with that color.  We note that our
  proof in fact shows that the graphs~$\mathcal{H}(F)$ in
  Theorem~\ref{thm:induced_subd} are $k$-partition induced universal
  for the classes of graphs~$H^\sigma$ described in that result.  An
  analogous remark holds for Theorem~\ref{thm:noninduced_subd}.
\end{remark}

\providecommand{\bysame}{\leavevmode\hbox to3em{\hrulefill}\thinspace}
\providecommand{\MR}{\relax\ifhmode\unskip\space\fi MR }
\providecommand{\arXiv}{\relax\ifhmode\unskip\space\fi arXiv }
\providecommand{\MRhref}[2]{%
  \href{http://www.ams.org/mathscinet-getitem?mr=#1}{#2}
}
\renewcommand{\MR}[1]{%
  \href{http://www.ams.org/mathscinet-getitem?mr=#1}{MR~#1}
}
\renewcommand{\arXiv}[1]{%
  Available as \href{https://arxiv.org/abs/#1}{arXiv:#1}.
}
\providecommand{\href}[2]{#2}

\appendix
\normalsize
\renewcommand\thetheorem{\Roman{theorem}}

\section{Proof of Theorem~\ref{thm:noninduced_subd}} \label{app}
Here, we give a sketch of the proof of Theorem~\ref{thm:noninduced_subd}. The proof is very similar to the proof of Theorem~\ref{thm:induced_subd}, we just need to modify the definition of goodness and critical set to match with the non-induced setting. We omit some proofs which are very similar to the induced setting. 

For the host graph, consider the hypergraph $\mathcal{H}'$ given in Lemma~\ref{lem:hypergraph} satisfying conditions (P1)--(P3) and (P4$'$) with parameters given in Table~\ref{tbl:parameters}. Also, we set $F$ as the gadget graph given in Table~\ref{tbl:gadgets} and take $\mathcal{H}'(F)$ as the host graph. Now, we define the auxiliary graph $G$ as follows (since for the general case, we take $F$ as a Ramsey graph for an odd cycle of unknown length, the auxiliary graph is a bit different with the induced case).

The vertex set of $G$ is the same as the vertex set of $\mathcal{H}'$. Fix an arbitrary $k$-edge coloring of $\mathcal{H}'(F)$. In the case of even subdivisions, the definition is similar to the induced setting. For each hyperedge of $\mathcal{H}'$, in its corresponding copy of $F$, there exists a monochromatic copy of $C_6$. We pick one of these cycles and add an edge in $G$ between two vertices at  distance two in the 6-cycle. Also, we color this edge in $G$ by the same color of the 6-cycle in $\mathcal{H}'(F)$.  

Now, for the case of general subdivisions, for each hyperedge of $\mathcal{H}'$, in its corresponding copy of $F$, there exists a monochromatic odd cycle of length at most $2\lceil \log_2 k\rceil +1$. Let $\ell$ be the most frequent length among these monochromatic cycles. Therefore, for each hyperedge whose corresponding copy of $F$ contains a monochromatic cycle $C_\ell$, we pick one of these cycles and add an edge in $G$ between two vertices at distance $(\ell-1)/2$ in the $\ell$-cycle in $F$. Also, we color this edge in $G$ by the same color of the $\ell$-cycle. In the case of even subdivision, we define $\ell=6$. We also define  $L_1=\lfloor (\ell-1)/2\rfloor $ and $L_2= \lceil (\ell+1)/2 \rceil  $.

The following lemma is the counterpart of Lemma~\ref{lem:subgraph} which elaborates the properties of the graph $G$.

\begin{lemma} \label{lem:subgraph_noninduced}
	The graph $G$ defined as above has the following properties. 
	\begin{itemize}
		\item[\rm (Q1)] $v(G)=v(\mathcal{H}')=N$ and $\Delta(G)\leq \Delta(\mathcal{H}')\leq 8cs$. Also, in the case of even subdivision, $e(G)=e(\mathcal{H}')\in [c N/2,c N]$ and in the general case, $e(G)\in [{c N}/(4\lceil \log_2 k\rceil +2),cN]$.
		\item[\rm (Q2)]  For any $A \subseteq V({G})$ with $|A|\leq \alpha N$, we have $e_G(A)\leq 2|A|$. 
		\item[\rm (Q3)] $G$ contains a monochromatic subgraph $G'$ such that 
		\begin{itemize}
			\item $v(G')\geq \alpha N$,
			\item the average degree of $G'$ is at least $500\,s^2D$, and
			\item $G'$ is a $\gamma$-expander graph, where $\gamma\geq (10^4sk\log_2 (sD))^{-1}$ for the even case and  $\gamma\geq (10^4sk\log_2k\log_2 (sD))^{-1}$ for the general case.
		\end{itemize} 
	\end{itemize}
\end{lemma}

We also change the definition of good subgraph where in the non-induced setting we drop the second condition in Definition~\ref{def:good}.
\begin{definition}
	A subgraph $J$ of $G$ is said to be a \textit{good subgraph} of $G$ (with respect to $\mathcal{H}'$) if 
	for every $e_1,e_2\in E(J)$, if $e_1$ and $e_2$ are  not adjacent, then $h(e_1)\cap h(e_2)=\emptyset$. 
	Also, an embedding $\phi:H\hookrightarrow G$ is good, if $\phi(H)$ is called a good subgraph of $G$.	
\end{definition}

Now, Lemma~\ref{lem:sigma'tosigma} is valid for the non-induced setting.

\begin{lemma}
	Suppose that $H^{\sigma'}$ is a monochromatic good subgraph of $G$ such that  $\sigma(e)/L_2\leq \sigma'(e)\leq \sigma(e)/L_1$, for all $e\in E(H)$. Then, $\mathcal{H}'(F)$ contains a monochromatic copy of $H^\sigma$ as a subgraph.
\end{lemma}

We also need to define available neighbors in the non-induced setting (see Definition~\ref{def:induced-available}).

\begin{definition}
	Let $J$ be a subgraph of $G'$. A neighbor $u$ of a vertex $v\in V(G')$ is called an \textit{available} neighbor of $v$ (with respect to $J$), if %
	$h(uv)\setminus\{v\}$ does not intersect $\cl(J)$. The set of all available neighbors of $v$ is denoted by $A_J(v)$.   
\end{definition}

The definition of critical set is the same as Definition~\ref{def:induced-critical}, by replacing the term `induced-available' with `available'. 
Given two positive integers $d,d'\geq 2$ and a subgraph $J$ of $G'$, the set of $(d,d')$-critical vertices of $J$ is denoted by $C_{d,d'}(J)$. 

Now, the non-induced counterpart of Lemma~\ref{lem:induced-critical} can be stated as follows.

\begin{lemma} \label{lem:critical}
	Let $d,d'\geq 2$ be two positive integers such that $d=  10\,sd'$. 
	Also, let $J$ be a subgraph of $G'$ where  $e(J)\leq \alpha N/(2s)-d'/2$. Also, let $X_0$ be the set of all non-leaves of $J$.  Then, $|C_{d,d'}(J)|\leq |X_0|+e(J)/d'$. 
\end{lemma}
The proof of Lemma~\ref{lem:critical} is similar to that of Lemma~\ref{lem:induced-critical}. We just have to apply (P4$'$) instead of (P4).
Finally, we need a variant of Lemma~\ref{lem:extend-induced} to extend good embeddings, which is as follows.

\begin{lemma} \label{lem:extend}
	Let $s\geq 3$ and $D\geq 2$ be integers and $K$ be a good subgraph of $G'$. Also, let $C\subseteq V(G')$ be a subset of vertices such that $|C|\leq \alpha N/(2sD)$. If for each vertex $v\in C$, $|A_K(v)|\geq 8sD$, then $K$ can be extended to a subgraph $\widehat{K}$, where $\widehat{K}$ is obtained from $K$ by adding some disjoint trees $T_1,\ldots, T_t$, where $t=|V(K)\cap C|$,  with roots in $V(K)\cap C$ such that
	\begin{itemize}
		\item $\widehat{K}$ is a good subgraph of $G'$.
		\item Every non-leaf vertex of $T_i$'s is in  $C$ and has $D$ children.
		\item No leaf of $T_i$'s is in $C$.  
	\end{itemize} 
\end{lemma}
The proof is similar to the proof of Lemma~\ref{lem:extend-induced}, but we apply (P4$'$) instead of (P4) and we use Corollary~\ref{cor:aharoni} by taking $D'=D$ (we do not need Lemma~\ref{lem:LLL}).

Now, we have all ingredients to prove Theorem~\ref{thm:noninduced_subd}. The trajectory of the proof is the same as the proof in Section~\ref{sec:mainproof}.   We are going to find $H^{\sigma'}$ as a good subgraph in $G'$. 
Note that, by Lemma~\ref{lem:subgraph_noninduced}, $G'$ has the average degree at least 
$500\, s^2D$, so there is an induced subgraph $G''$ of $G'$ with $\delta(G'')\geq 250\, s^2D$.

We follow the proof by setting $d'=8\, sD$ and $d=80\, s^2D$ and parameters $N,\alpha,c,g$ as in Table~\ref{tbl:parameters} (the non-induced column). Everywhere we use Lemmas~\ref{lem:critical} and \ref{lem:extend} instead of Lemmas~\ref{lem:induced-critical} and \ref{lem:extend-induced}, respectively.  Also, in Phase~2, we use $\cl_2(K_i)$ instead of $\cl_3(K_i)$. 

This concludes our sketch of the proof of Theorem~\ref{thm:noninduced_subd}.

\end{document}